\documentclass[11pt]{amsart}

\usepackage{amsmath}
\usepackage{amssymb}
\usepackage{hyperref}
\usepackage{eucal}
\usepackage{upgreek}
\usepackage{mathtools}

\usepackage{enumitem}

\usepackage{xy}
\xyoption{all}

\newcommand{\nc}{\newcommand}

\nc{\CC}{{\mathbb{C}}}
\renewcommand{\P}{{\mathbb{P}}}
\nc{\ZZ}{{\mathbb{Z}}}
\nc{\Z}{{\mathbb{Z}}}

\nc{\cA}{{\mathcal{A}}}
\nc{\cB}{{\mathcal{B}}}
\nc{\cC}{{\mathcal{C}}}
\nc{\cD}{{\mathcal{D}}}
\nc{\cE}{{\mathcal{E}}}
\nc{\cF}{{\mathcal{F}}}
\nc{\bcF}{\bar\cF}
\nc{\cG}{{\mathcal{G}}}
\nc{\cI}{{\mathcal{I}}}
\nc{\cJ}{{\mathcal{J}}}
\nc{\cL}{{\mathcal{L}}}
\nc{\cM}{{\mathcal{M}}}
\nc{\cN}{{\mathcal{N}}}
\nc{\cO}{{\mathcal{O}}}
\nc{\cQ}{{\mathcal{Q}}}
\nc{\cS}{{\mathcal{S}}}
\nc{\cT}{{\mathcal{T}}}
\nc{\cX}{{\mathcal{X}}}

\newcommand{\tcB}{\widetilde\cB}

\nc{\cXo}{X}
\nc{\bcXo}{\overline{X}}
\nc{\hcXo}{\hX}
\nc{\rhoo}{\rho}
\nc{\crho}{\rho_\cX}

\nc{\rE}{{\mathcal{E}}}
\nc{\rG}{{\mathrm{G}}}
\nc{\urG}{\underline{\rG}}
\nc{\rH}{{\mathrm{H}}}
\nc{\rK}{{\mathrm{K}}}
\nc{\rL}{{\mathrm{L}}}

\nc{\bD}{{\mathbf{D}}}
\nc{\bE}{{\mathbf{E}}}
\nc{\bF}{{\mathbf{F}}}
\nc{\bL}{{\mathbf{L}}}
\nc{\bR}{{\mathbf{R}}}
\nc{\bS}{{\mathbf{S}}}
\nc{\bT}{{\mathbf{T}}}

\nc{\bd}{{\mathbf{d}}}
\nc{\bq}{{\mathbf{q}}}

\nc{\fa}{{\mathfrak{a}}}

\nc{\tX}{{\widetilde{X}}}

\nc{\tcD}{{\widetilde{\cD}}}

\nc{\hX}{{\widehat{X}}}

\nc{\eps}{\upepsilon}
\nc{\kk}{{\Bbbk}}
\nc{\io}{\upiota}

\nc{\Db}{\bD^{\mathrm{b}}}
\usepackage{accents}
\nc{\RPG}[1]{\accentset{\bullet}\bD^{\mathrm{b}}(#1)}

\newcommand{\opp}{\mathrm{op}}
\newcommand{\tp}{\mathrm{top}}

\DeclareMathOperator{\Aut}{\mathrm{Aut}}
\DeclareMathOperator{\Hom}{\mathrm{Hom}}
\DeclareMathOperator{\Ext}{\mathrm{Ext}}
\DeclareMathOperator{\RHom}{\mathrm{RHom}}
\DeclareMathOperator{\Tor}{\mathrm{Tor}}
\DeclareMathOperator{\gr}{gr}

\DeclareMathOperator{\Spec}{\mathrm{Spec}}
\DeclareMathOperator{\Bl}{\mathrm{Bl}}
\DeclareMathOperator{\Pic}{\mathrm{Pic}}
\DeclareMathOperator{\Cliff}{\mathcal{C}\!\ell}
\DeclareMathOperator{\Stab}{\mathrm{Stab}}

\DeclareMathOperator{\Sym}{\mathrm{Sym}}
\DeclareMathOperator{\Ker}{\mathrm{Ker}}
\DeclareMathOperator{\Cone}{\mathrm{Cone}}

\DeclareMathOperator{\pr}{\mathrm{pr}}

\DeclareMathOperator{\Gr}{\mathrm{Gr}}

\DeclareMathOperator{\g}{\mathrm{g}}
\DeclareMathOperator{\ev}{\mathrm{ev}}
\DeclareMathOperator{\id}{\mathrm{id}}
\DeclareMathOperator{\pt}{\mathrm{pt}}

\DeclareMathOperator{\At}{\mathrm{At}}

\DeclareMathOperator{\Jac}{Jac}
\DeclareMathOperator{\HH}{\mathsf{HH}}
\DeclareMathOperator{\rKn}{\mathrm{\rK}_0^{\mathrm{num}}}

\nc{\bcX}{\overline{\cX}}

\numberwithin{equation}{section}

\theoremstyle{plain}

\newtheorem{theorem}{Theorem}[section]
\newtheorem{lemma}[theorem]{Lemma}
\newtheorem{proposition}[theorem]{Proposition}
\newtheorem{corollary}[theorem]{Corollary}

\newtheorem*{claim*}{Claim}

\theoremstyle{definition}

\newtheorem{definition}[theorem]{Definition}

\theoremstyle{remark}

\newtheorem{remark}[theorem]{Remark}
\newtheorem{notation}[theorem]{Notation}

\newenvironment{arenumerate}{\begin{enumerate}[label={\textup{(\arabic*)}}]}{\end{enumerate}}

\newcommand{\xrightiso}{ \xrightarrow{\ \raisebox{-0.5ex}[0ex][0ex]{$\sim$}\ } }

\title[Augmentations, reduced ideal point gluings and compact type degenerations of curves]%
{Augmentations, reduced ideal point gluings\\[1ex]and compact type degenerations of curves}

\author{Valery Alexeev}
\address{Department of Mathematics, University of Georgia, Athens GA 30602, USA}
\email{valery@math.uga.edu}

\author{Alexander Kuznetsov}
\address{{\sloppy
\parbox{0.99\textwidth}{
Algebraic Geometry Section, Steklov Mathematical Institute of Russian Academy of Sciences\\
8 Gubkin str., Moscow 119991 Russia
\\[3pt]
Laboratory of Algebraic Geometry, NRU Higher School of Economics, Russian Federation
}\medskip}}
\email{akuznet@mi-ras.ru}



\begin{document}

\begin{abstract}
In this note we demonstrate some unexpected properties that simple gluings of the simplest derived categories may have.
We consider two special cases:
the first is an {\sf augmented curve}, 
i.e., the gluing of the derived categories of a point and a curve 
with the gluing bimodule given by the structure sheaf of the curve;
the second is an {\sf ideal point gluing of curves}, 
i.e., the gluing of the derived categories of two curves
with the gluing bimodule given by the ideal sheaf of a point in the product of the curves.
We construct unexpected exceptional objects contained in these categories
and discuss their orthogonal complements.

We also show that the simplest example of compact type degeneration of curves,
a flat family of curves with a smooth general fiber 
and a 1-nodal reducible central fiber,
gives rise to a smooth and proper family of triangulated categories
with the general fiber an augmented curve and the central fiber 
the orthogonal complement of the exotic exceptional object in the ideal point gluing of curves,
called the {\sf reduced ideal point gluing of curves}.
\end{abstract}

\maketitle

\section{Introduction}

If~$\cD_1$ and~$\cD_2$ are appropriately enhanced triangulated categories over a field~$\kk$ and
\begin{equation*}
\rG \colon \cD_1^\opp \times \cD_2 \to \bD(\kk)
\end{equation*}
is a bimodule, one can construct a new triangulated category~$\cD$ that has a semiorthogonal decomposition
\begin{equation*}
\cD = \langle \cD_1, \cD_2 \rangle
\end{equation*}
with $\Hom$-spaces from the objects of~$\cD_1$ to the objects of~$\cD_2$
given by the bimodule~$\rG$, i.e., so that
\begin{equation*}
\RHom_\cD(\cF_1,\cF_2) \cong \rG(\cF_1,\cF_2)
\end{equation*}
for any~$\cF_i \in \cD_i \subset \cD$, see~\cite[\S4]{KL15} or Section~\ref{sec:gluing} for a reminder.
The category~$\cD$ with the above properties is unique up to equivalence;
it is called {\sf the gluing} of~$\cD_1$ and~$\cD_2$ with respect to~$\rG$.

The simplest example of a gluing is obtained if~$\cD_1 = \cD_2 = \Db(\kk)$.
In this case a bimodule~$\rG$ is determined by its value~$\rG(\kk,\kk)$ on the generators of~$\cD_1$ and~$\cD_2$, 
which is just a single complex of vector spaces.
The corresponding gluing is the derived category of a generalized Kronecker quiver,
and if~$\rG(\kk,\kk)$ is just a vector space of dimension~$n$, this is the usual Kronecker quiver with~$n$ arrows.
More generally, the derived category of any directed quiver can be realized as an iterated gluing of~$\Db(\kk)$.

In this paper we study slightly more complicated examples of geometric gluings, 
where one or both of~$\cD_i$ is the derived category of a smooth proper curve
and the gluing bimodule is chosen appropriately.

\subsection{Augmented curves}

The first example, where
\begin{equation*}
\cD_1 = \Db(\kk),
\qquad
\cD_2 = \Db(C),
\qquad\text{and}\qquad 
\rG(\kk, \cF) = \rH^\bullet(C,\cF)
\end{equation*}
for a smooth proper curve~$C$, is called the {\sf augmentation} of~$C$; we denote it by~$\Db(\cO,C)$.

Equivalently, $\Db(\cO,C)$ is a triangulated category with a semiorthogonal decomposition
\begin{equation}
\label{eq:intro-dbac}
\Db(\cO,C) = \langle \cE, \Db(C) \rangle,
\end{equation}
where~$\cE$ is an exceptional object and~$\RHom_{\Db(\cO,C)}(\cE,\cF) \cong \RHom_{\Db(C)}(\cO_C, \cF)$ for~$\cF \in \Db(C)$.

The augmentation~$\Db(\cO,C)$ of a curve~$C$ appears naturally as an admissible subcategory 
of the derived category of the blowup of a smooth and proper variety with exceptional structure sheaf,
see Lemma~\ref{lem:ac-blowup}.

We study augmentations in Section~\ref{sec:ac}.
First of all, we compute the basic invariants of augmented curves:
their Hochschild homology and cohomology, the (numerical) Grothendieck group, and the intermediate Jacobian,
see Proposition~\ref{prop:invariants-ac}.
We also describe the Serre functor of augmented curves (Theorem~\ref{thm:serre-ac}). 

Furthermore, for any coherent sheaf~$\cF$ on~$C$ 
we consider the natural object~$\fa(\cF) \in \Db(\cO,C)$ (called
the {\sf augmentation} of~$\cF$)
composed from the vector space~$\rH^0(C,\cF)$ and the sheaf~$\cF$,
see Definition~\ref{def:augmented-sheaf},
and show that augmentations of line bundles with special Brill--Noether properties 
have interesting categorical behavior.
In particular, if~$\cL$ is a line bundle on~$C$ such that~$h^0(\cL) \cdot h^1(\cL) = \g(C)$ 
and, moreover, the Petri map~$\rH^0(C,\cL) \otimes \rH^0(C, \cL^\vee(K_C)) \to \rH^0(C, \cO_C(K_C))$ is an isomorphism,
we show that the object
\begin{equation*}
\cE_\cL \coloneqq \fa(\cL) \in \Db(\cO,C)
\end{equation*}
is exceptional, see Proposition~\ref{prop:bnp-exceptional}.
We call such objects {\sf BN-exceptional}, see Definition~\ref{def:bn-exceptional}.

In the simplest cases, \mbox{$\cL = \cO_C$} or~\mbox{$\cL = \omega_C$}, 
the corresponding exceptional objects~$\cE_\cO$ and~$\cE_\omega$ in~$\Db(\cO,C)$
can be obtained from the canonical exceptional object~\mbox{$\cE \in \Db(\cO,C)$} from~\eqref{eq:intro-dbac}
by an autoequivalence, see Remark~\ref{rem:bnex-o-om};
in particular, their orthogonal complements in~$\Db(\cO,C)$ are equivalent to~$\Db(C)$.

Otherwise, if~$\cL \ne \cO_C$, $\omega_C$, the category~${}^\perp\cE_\cL$, 
though it has the same Hochschild homology and intermediate Jacobian as~$\Db(C)$ (see Proposition~\ref{prop:invariants-mac}),
is not equivalent to the derived category of any curve (Corollary~\ref{cor:bnex-ort}).
We call this category the {\sf $\cL$-} or {\sf BN-modification} of~$\Db(C)$, see Definition~\ref{def:bn-modification}.

The simplest example of a nontrivial BN-modification of a curve appears 
when~$C$ is a general curve of genus~$\g(C) = 4$ and~$\cL$ is a trigonal line bundle.
We show that in this case the~$\cL$-modification of~$\Db(C)$ 
provides a crepant categorical resolution for the orthogonal complement of~$\cO(-1)$ and~$\cO$
in the derived category of a cubic threefold with one node, see Appendix~\ref{sec:cubic-3}.

We also give an alternative description of augmented curves of small genus;
in particular, we show that
\begin{itemize}
\item 
if~$\g(C) = 0$, then~$\Db(\cO,C)$ is the derived category of a quiver with no relations, see Lemma~\ref{lem:ac-0};
\item 
if~$\g(C) = 1$, then~$\Db(\cO,C)$ is the derived category of a stacky curve, see Lemma~\ref{lem:ac-1}.
\end{itemize}
In particular, these categories have Serre dimension~$1$.
In contrast to these two cases, we prove that
\begin{itemize}
\item 
if~$\g(C) = 5$, then~$\Db(\cO,C)$ is a twisted derived category of a stacky surface, see Proposition~\ref{prop:ac-g5}.
\end{itemize}
In fact, we expect that~$\Db(\cO,C)$ behaves as the derived category of a surface for (general) curves~$C$ of genus~$\g(C) \ge 5$;
for instance, in~\cite{AK2} we will show that in this case the Serre dimension of~$\Db(\cO,C)$ equals~2
(while for~$g = 2,3,4$ the lower and upper Serre dimensions of~$\Db(\cO,C)$
are equal to~$(4/3,4/3)$, $(3/2,2)$, and~$(5/3,2)$, respectively).
The corresponding noncommutative surface can be considered to be the \emph{canonical surface} containing~$C$.
Following this point of view, we suggest to consider augmentations~$\fa(\cL) \in \Db(\cO,C)$ of line bundles on~$C$
as analogs of \emph{Lazarsfeld bundles}, see~\cite{lazarsfeld1986brill-noether}. 

We are confident that augmented curves form a very interesting and useful class of triangulated categories.
The present discussion is just the very first step towards their
study, which we plan to continue in follow-up papers.
Among other things, understanding the space~$\Stab(\Db(\cO,C))$ of stability conditions on~$\Db(\cO,C)$ is an interesting problem;
an important step in this direction is accomplished in~\cite{FN}, where an open subset of~$\Stab(\Db(\cO,C))$ is studied.

\subsection{Ideal point gluing of curves}

Our second example of gluing is
\begin{equation*}
\cD_1 = \Db(C_1),
\qquad
\cD_2 = \Db(C_2),
\qquad\text{and}\qquad 
\rG(\cF_1, \cF_2) = \rH^\bullet(C_1 \times C_2,(\cF_1^\vee \boxtimes \cF_2) \otimes \cI_{(x_1,x_2)}),
\end{equation*}
where~$C_1$, $C_2$ are smooth proper curves, $x_1 \in C_1$, $x_2 \in C_2$ are $\kk$-points,
and~$\cI_{(x_1,x_2)}$ is the ideal sheaf of the point~$(x_1,x_2) \in C_1 \times C_2$.
Accordingly, we call this category {\sf the ideal point gluing} of~$C_1$ and~$C_2$
and denote it by~$\Db(C_1,C_2)$.

Equivalently, $\Db(C_1,C_2)$ is a triangulated category with a semiorthogonal decomposition
\begin{equation}
\label{eq:intro-ipg}
\Db(C_1,C_2) = \langle \Db(C_1), \Db(C_2) \rangle,
\end{equation}
such that~$\RHom_{\Db(C_1,C_2)}(\cF_1,\cF_2) \cong 
\rH^\bullet(C_1 \times C_2,(\cF_1^\vee \boxtimes \cF_2) \otimes \cI_{(x_1,x_2)})$.

The ideal point gluing 
can be realized as an admissible subcategory in the derived category 
of a blowup of two transverse curves in a threefold with exceptional structure sheaf,
see Proposition~\ref{prop:gluing-geometric}.

We study the ideal point gluing in Section~\ref{sec:rip-gluing}.
As in the case of augmentations, 
we start by computing the basic invariants of~$\Db(C_1,C_2)$ (see Proposition~\ref{prop:invariants-pg}).
Then, we present the most unexpected property of this category ---  
the existence of an {\sf exotic exceptional object}~$\rE \in \Db(C_1,C_2)$
with components~\mbox{$\cO_{x_i} \in \Db(C_i)$} (up to a shift), see Theorem~\ref{thm:exotic-exceptional-curves} for the proof 
and Remark~\ref{rem:generalization} for a generalization. 

After that, we consider the orthogonal complement of the object~$\rE$
\begin{equation*}
\RPG{C_1,C_2}
\coloneqq {}^\perp\rE \subset \Db(C_1,C_2),
\end{equation*}
which we call {\sf the reduced ideal point gluing} of~$C_1$ and~$C_2$.
We compute invariants of~$\RPG{C_1,C_2}$ (see Proposition~\ref{prop:invariants-rpg}) and observe that
it is numerically equivalent to an augmented curve of genus~\mbox{$\g(C_1) + \g(C_2)$}.
Moreover, if~$\g(C_1) = 0$, this category is equivalent to the augmentation of~$C_2$, 
and if~$\g(C_2) = 0$, it is equivalent to the augmentation of~$C_1$, see Remark~\ref{rem:rpg-zero}.
However, if~$\g(C_1) > 0$ and~$\g(C_2) > 0$, it is not equivalent to an augmentation, see Remark~\ref{rem:rpg-nonzero}.

\subsection{Compact type degenerations of curves}

In the last section we explain the relation between reduced ideal point gluings and augmentations of curves.

Let~$\cC \to B$ be a flat proper morphism from a smooth surface~$\cC$ to a smooth pointed curve~$(B,o)$
(which we consider as a family of curves~$\cC_b$ parameterized by~$B$)
such that the central fiber
\begin{equation*}
\cC_o = C_1 \cup C_2
\end{equation*}
is a 1-nodal curve with two smooth components~$C_1$ and~$C_2$
and all other fibers~$\cC_b$ are smooth.
The main result of this section is a construction in Theorem~\ref{thm:main} 
of a family of $B$-linear triangulated categories~$\cD/B$ smooth and proper over~$B$ such that
\begin{equation*}
\cD_o \simeq \RPG{C_1,C_2}
\qquad\text{and}\qquad 
\cD_b \simeq \Db(\cO, \cC_b).
\end{equation*}
In other words, we show that the reduced ideal point gluing of the components of the central fiber of~$\cC/B$
is a smooth degeneration of the augmentations of its general fibers.

To prove this theorem, we embed the family of curves~$\cC/B$ into a family of smooth
and proper threefolds~$\bcX/B$ such that~$\cO_{\bcX_b}$ is exceptional for all~$b \in B$
and consider the blowup~$\cX \coloneqq \Bl_{\cC}(\bcX)$,
so that~$\cX/B$ is a smoothing of the 1-nodal nonfactorial threefold
\begin{equation*}
\cX_o \cong \Bl_{C_1 \cup C_2}(\bcX_o).
\end{equation*} 
Applying the technique of categorical absorption of singularities developed in~\cite{KS24},
we construct a~$\P^{\infty,2}$-object in the derived category of the central fiber,
so that its pushforward to~$\Db(\cX)$ is exceptional and the orthogonal complement
of this exceptional object is a family of triangulated categories smooth and proper over~$B$.
Finally, considering the orthogonal complements
of the subcategories~$\cO_{\bcX_b}^\perp \subset \Db(\bcX_b) \subset \Db(\cX_b)$ for all points~$b \in B$
(recall that~$\cO_{\bcX_b}$ is an exceptional line bundle for each~$b \in B$),
we obtain the required family of $B$-linear triangulated categories~$\cD/B$.

We expect that the above approach can be used 
to construct smooth and proper families of triangulated categories as in Theorem~\ref{thm:main}
for more complicated compact type degenerations of curves.

We work over any field~$\kk$.
We write~$\bD(X)$ and~$\Db(X)$ for the unbounded derived category of quasicoherent sheaves 
and the bounded derived category of coherent sheaves.
All functors are derived.
Given a functor~$\Phi$, we denote by~$\Phi^*$ and~$\Phi^!$ its left and right adjoints;
$\bS_\cD$ denotes the Serre functor of~$\cD$.

\subsection*{Acknowledgments}

We thank Soheyla Feyzbakhsh and Alex Perry for useful discussions
and the referee for their comments.
V.A.\ was partially supported by the NSF under DMS-2501855.
A.K.\ was partially supported by the Basic Research Program of HSE University (HSE-BR-2025-061).

\section{Generalities about gluing}
\label{sec:gluing}

In this section we recall the general construction of a gluing and discuss a few general properties.

We say that a triangulated category~$\cD$ is {\sf a gluing} of triangulated categories~$\cD_1$ and~$\cD_2$
if there exists a semiorthogonal decomposition
\begin{equation*}
\cD = \langle \cD_1, \cD_2 \rangle.
\end{equation*}
The main characteristic of a gluing is the {\sf gluing bimodule}
\begin{equation}
\label{eq:rg-rhom}
\rG \colon \cD_1^\opp \times \cD_2 \to \bD(\kk),
\qquad 
\rG(\cF_1,\cF_2) \coloneqq \RHom_\cD(\cF_1,\cF_2),
\qquad 
\text{where~$\cF_i \in \cD_i$}.
\end{equation}
In the geometric case, i.e., if~$\cD_i = \Db(X_i)$ for some smooth and proper algebraic varieties~$X_1$ and~$X_2$,
every bimodule~$\rG$ is representable by an object~$\urG \in \bD(X_1 \times X_2)$, so that
\begin{equation*}
\rG(\cF_1,\cF_2) \cong \rH^\bullet(X_1 \times X_2, (\cF_1^\vee \boxtimes \cF_2) \otimes \urG).
\end{equation*}
In this case we say that~$\urG$ is the {\sf gluing object}.

If the categories~$\cD_i$ are appropriately enhanced, 
the gluing can be reconstructed from the components and the gluing bimodule.
We will use the following version of this result:

\begin{theorem}[{\cite[Section~4]{KL15}}]
\label{thm:gluing}
Let~$\cD_1$ and~$\cD_2$ be dg-enhanced small triangulated categories 
and let~$\rG \colon \cD_1^\opp \times \cD_2 \to \bD(\kk)$ be a dg-enhanced bimodule.
Then there is, up to equivalence, a unique dg-enhanced triangulated category~$\cD_1 \times_\rG \cD_2$ 
which has the structure of a gluing of~$\cD_1$ and~$\cD_2$ with the gluing bimodule~$\rG$,
i.e., there is a semiorthogonal decomposition~$\cD_1 \times_\rG \cD_2 = \langle \cD_1, \cD_2 \rangle$ 
and~\eqref{eq:rg-rhom} holds.

Let~$\cD_1 \times_\rG \cD_2$ and~$\cD'_1 \times_{\rG'} \cD'_2$ be two categories obtained by gluing.
If~$\Phi_1 \colon \cD_1 \to \cD'_1$ and~$\Phi_2 \colon \cD_2 \to \cD'_2$ are dg-enhanced functors 
compatible with the gluing bimodules, i.e., there is an isomorphism of bimodules
\begin{equation*}
\rG' \circ (\Phi_1^\opp \times \Phi_2) \cong \rG \colon \cD_1^\opp \times \cD_2 \to \bD(\kk),
\end{equation*}
then there is a functor~$\Phi \colon \cD_1 \times_\rG \cD_2 \to \cD'_1 \times_{\rG'} \cD'_2$
which restricts to~$\Phi_i$ on~$\cD_i$.
If both~$\Phi_1$ and~$\Phi_2$ are fully faithful, or are equivalences, then so is~$\Phi$.
\end{theorem}

\begin{proof}
The construction of the triangulated dg-category~$\cD_1 \times_\rG \cD_2$ can be found in~\cite[Section~4.1]{KL15},
and the semiorthogonal decomposition is constructed in~\cite[Section~4.2]{KL15}.

The second part of the theorem follows from~\cite[Proposition~4.11]{KL15} and its argument.
In turn, it implies the uniqueness of the gluing.
\end{proof}

\begin{corollary}
\label{cor:gluing-auto}
If~$\cD_i = \Db(X_i)$ and two gluing objects~$\urG, \urG' \in \Db(X_1 \times X_2)$ 
are related by twists and a shift, 
i.e., $\urG' = \urG \otimes (\cL_1 \boxtimes \cL_2)[m]$ for~$\cL_i \in \Pic(X_i)$ and~$m \in \ZZ$,
then there is an equivalence
\begin{equation*}
\Phi \colon \Db(X_1) \times_\rG \Db(X_2) \xrightiso \Db(X_1) \times_{\rG'} \Db(X_2)
\end{equation*}
whose restrictions to~$\Db(X_1)$ and~$\Db(X_2)$ are isomorphic to the twists by~$\cL_1[m]$ and~$\cL_2^{-1}$, respectively.
\end{corollary}

\begin{remark}
\label{rem:gluing-dual}
If a smooth and proper category~$\cD$ has a semiorthogonal decomposition~$\cD = \langle \cD_1, \cD_2 \rangle$,
it also has a semiorthogonal decomposition~$\cD = \langle \cD_2, \cD_1 \rangle$, obtained from the first one by mutation.
It is easy to check that if the gluing bimodule for the first gluing
is~$\rG$, then for the second one it is~$\rG^\vee$.
In particular, $\cD_1 \times_\rG \cD_2 \simeq \cD_2 \times_{\rG^\vee} \cD_1$.
\end{remark}

\begin{notation}
\label{not:fff}
Let~$\cD = \langle \cD_1, \cD_2 \rangle$ be a gluing with gluing bimodule~$\rG$.
For any~$\cF_1 \in \cD_1$, $\cF_2 \in \cD_2$, and~$\phi \in \Hom_{\cD}(\cF_1,\cF_2) = \rH^0(\rG(\cF_1,\cF_2))$ 
we denote by
\begin{equation*}
(\cF_1,\cF_2,\phi) \coloneqq \Cone\Big(\cF_1 \xrightarrow{\ \phi\ } \cF_2\Big)[-1],
\end{equation*}
the object in~$\cD$ that fits into a distinguished
triangle~$(\cF_1,\cF_2,\phi) \xrightarrow{\quad} \cF_1 \xrightarrow{\ \phi\ } \cF_2$.
\end{notation}

\begin{remark}
\label{rem:gluing-components}
One should keep in mind that the components of~$\cF = (\cF_1,\cF_2,\phi)$
with respect to the semiorthogonal decomposition are~$\cF_1$ and~$\cF_2[-1]$
(not~$\cF_1$ and~$\cF_2$, as one could expect).
\end{remark}

To compute $\Ext$-spaces in the gluing, it is convenient to use the following

\begin{lemma}
\label{lem:hom-gluing}
If~$\cD = \langle \cD_1, \cD_2\rangle$ is a gluing with gluing bimodule~$\rG$, 
for any objects~$\cF = (\cF_1,\cF_2,\phi)$ and~$\cF' = (\cF'_1,\cF'_2,\phi')$ in~$\cD$ there is a distinguished triangle 
\begin{equation*}
\RHom_\cD(\cF,\cF') \to \RHom_{\cD_1}(\cF_1,\cF'_1) \oplus \RHom_{\cD_2}(\cF_2,\cF'_2) \to \rG(\cF_1,\cF'_2)
\end{equation*}
in~$\bD(\kk)$, with the second map defined by~$(f_1,f_2) \mapsto \phi' \circ f_1 - f_2 \circ \phi$ 
for all~$f_i \in \RHom_{\cD_i}(\cF_i,\cF'_i)$.
\end{lemma}

\begin{proof}
This follows immediately from the construction of the dg-category~$\cD_1 \times_\rG \cD_2$ in~\cite{KL15}
(see~\cite[Remark~4.1]{KL15}).
\end{proof}

\subsection{Basic properties of the gluing}
\label{ss:gluing-basics}

Recall that the {\sf diagonal bimodule} of an enhanced triangulated category~$\cD$ is defined by
\begin{equation*}
\cD(\cF_1,\cF_2) \coloneqq \RHom_\cD(\cF_1,\cF_2).
\end{equation*}
Recall also that a triangulated category~$\cD$ is {\sf proper} over~$\kk$ 
if the graded $\kk$-vector spaces~$\Ext_\cD^\bullet(-,-)$ are finite-dimensional
(if the category is enhanced, this can be rephrased by saying that the diagonal bimodule of~$\cD$ is perfect over~$\kk$),
and an enhanced triangulated category~$\cD$ is {\sf smooth} over~$\kk$ 
if its diagonal bimodule is perfect over~$\cD^\opp \otimes_\kk \cD$.

\begin{proposition}[{\cite[Proposition~3.22 and Theorem~3.25]{O16}}]
\label{prop:gluing-sp}
Let~$\cD = \langle \cD_1, \cD_2\rangle$ be a gluing with gluing bimodule~$\rG$.
\begin{arenumerate}
\item 
The category~$\cD$ is proper over~$\kk$ if and only if~$\cD_1$ and~$\cD_2$ are proper and~$\rG$ is perfect over~$\kk$.
\item 
The category~$\cD$ is smooth over~$\kk$ if and only if~$\cD_1$ and~$\cD_2$ are smooth and~$\rG$ is perfect over~$\kk$.
\end{arenumerate}
\end{proposition}

If~$\cD_1$ and~$\cD_2$ are both smooth and proper, the assumption that~$\rG$ is perfect over~$\kk$ 
just means that~$\rG$ takes values in~$\Db(\kk) \subset \bD(\kk)$, see~\cite[Lemma~3.8]{KS-hfd}.

The next proposition shows how to compute some invariants of the gluing.
Recall that {\sf the Euler form} on the Grothendieck group~$\rK_0(\cD)$ of a smooth and proper triangulated category~$\cD$
is defined by
\begin{equation*}
\upchi_\cD([\cF_1], [\cF_2]) \coloneqq \sum(-1)^i \dim \Ext_\cD^i(\cF_1, \cF_2).
\end{equation*}
Furthermore, the numerical Grothendieck group~$\rKn(\cD)$ is defined as the quotient of~$\rK_0(\cD)$
by the kernel of the Euler form; obviously, the Euler form descends to this quotient group.

\begin{proposition}
\label{prop:invariants-general}
Let~$\cD = \langle \cD_1, \cD_2\rangle$ be a smooth and proper gluing with gluing bimodule~$\rG$.
\begin{arenumerate}
\item 
\label{it:additivity}
There are direct sum decompositions for Hochschild homology and K-theory
\begin{align*}
\HH_\bullet(\cD) &= \HH_\bullet(\cD_1) \oplus \HH_\bullet(\cD_2),\\
\rK_\bullet(\cD) &= \rK_\bullet(\cD_1) \oplus \rK_\bullet(\cD_2).
\end{align*}
\item 
\label{it:rkn}
There is a semiorthogonal direct sum decomposition of the numerical Grothendieck group
\begin{equation*}
\rKn(\cD) = \rKn(\cD_1) \oplus \rKn(\cD_2),
\qquad 
\upchi_\cD = 
\begin{pmatrix}
\upchi_{\cD_1} & [\rG] \\ 0 & \upchi_{\cD_2}
\end{pmatrix},
\end{equation*}
where~$[\rG]$ stands for the pairing defined by~$[\rG]([\cF_1], [\cF_2]) \coloneqq \sum (-1)^i \dim \rH^i(\rG(\cF_1,\cF_2))$.
\item 
\label{it:hh-triangle}
There is a distinguished triangle for Hochschild cohomology
\begin{equation*}
\HH^\bullet(\cD) \to \HH^\bullet(\cD_1) \oplus \HH^\bullet(\cD_2) \to \Ext^\bullet_{\cD_1 \otimes \cD_2^\opp}(\rG, \rG),
\end{equation*}
where the last term is computed in the category of bimodules
and the last arrow is induced by the action of~$\HH^\bullet(\cD_i)$,
considered as endomorphisms of the identity functor of~$\cD_i$.
\end{arenumerate}
\end{proposition}

\begin{proof}
Part~\ref{it:additivity} follows from additivity, see, e.g., \cite[Corollary~7.5]{K09-HH},
and part~\ref{it:hh-triangle} is proved in~\cite[Theorem~7.7 and Remark~7.8]{K09-HH}.
To prove part~\ref{it:rkn} we use the embedding functors~$\cD_i \hookrightarrow \cD$
to induce the maps~$\rKn(\cD_i) \hookrightarrow \rKn(\cD)$
and Lemma~\ref{lem:hom-gluing} and Remark~\ref{rem:gluing-components} to compute the Euler characteristic.
\end{proof}

Recall from~\cite[Definition 5.24]{Per22} the definition of the {\sf intermediate Jacobian}~$\Jac(\cD)$
of a triangulated category~$\cD$ of geometric origin, 
i.e., of an admissible subcategory of the derived category of a
smooth and proper complex variety~$X$.
There is a general intrinsic construction~\cite[Proposition~5.4]{Per22}
that conjecturally endows the odd topological $\rK$-group~$\rK_1^\tp(\cD)$ 
associated to a $\CC$-linear triangulated category~$\cD$
with an integral Hodge structure of weight~$-1$, and~$\Jac(\cD)$ is defined as the corresponding complex torus.
When~$\cD$ is an admissible subcategory in~$\Db(X)$, the conjecture is known to hold, so~$\Jac(\cD)$ is well defined; 
moreover, it is a subtorus in the torus~$\Jac(X) \coloneqq \Jac(\Db(X))$, 
isogenous to the product of the Jacobians of~$X$ in all degrees.

\begin{remark}
\label{rem:ppav}
The group~$\rK_1^\tp(\cD)$ is endowed with a pairing~$\upchi^\tp_\cD \colon \rK_1^\tp(\cD) \otimes \rK_1^\tp(\cD) \to \ZZ$,
called {\sf the Euler pairing} (see~\cite[Lemma~5.2]{Per22});
moreover, if~$\cD = \langle \cD_1, \cD_2 \rangle$ is a semiorthogonal decomposition,
the restriction of~$\upchi^\tp_\cD$ to~$\rK_1^\tp(\cD_i)$ coincides with~$\upchi^\tp_{\cD_i}$,
and there is a direct sum decomposition
\begin{equation*}
\rK_1^\tp(\cD) = \rK_1^\tp(\cD_1) \oplus \rK_1^\tp(\cD_2)
\end{equation*}
semiorthogonal with respect to~$\upchi^\tp_\cD$.
If~$\cD = \Db(X)$, where~$X$ is a smooth and proper complex variety of odd dimension~$n$, 
it is shown in~\cite[Proposition~5.23]{Per22} that under appropriate assumptions about~$X$
the Chern character induces an isomorphism~$\rK_1^\tp(\cD) \cong \rH^n(X,\ZZ)$
and identifies the Euler pairing of the source with the cup-product of the target;
in particular, in this case the Euler pairing is unimodular and skew-symmetric.
If, moreover, $\rH^n(X, \CC)$ has the Hodge decomposition~\mbox{$\rH^{p,q}(X) \oplus \rH^{q,p}(X)$} for some~\mbox{$p + q = n$},
the Euler form is also positive definite (more precisely, the corresponding hermitian form on~$\rH^n(X, \CC)$ is positive definite).
Therefore, the same is true for any semiorthogonal component~$\cD \subset \Db(X)$,
and in particular in this case~$\Jac(\cD)$ is a principally polarized abelian variety.
\end{remark}

\begin{proposition}
\label{prop:ij-gluing}
Let~$\cD = \langle \cD_1, \cD_2\rangle$ be a gluing with gluing bimodule~$\rG$.
If~$\cD_1$ and~$\cD_2$ are admissible subcategories of the derived categories of smooth projective complex varieties
and the gluing bimodule~$\rG$ is perfect over~$\CC$, then~$\Jac(\cD)$ is well defined 
and~$\Jac(\cD) \cong \Jac(\cD_1) \times \Jac(\cD_2)$.
\end{proposition}

\begin{proof}
Since~$\cD_i$ are admissible subcategories of the derived categories of smooth projective varieties,
the same is true for the category~$\cD$ by~\cite[Theorem~4.15]{O16}, hence~$\Jac(\cD)$ is well defined.
Since, moreover, the definition of~$\Jac(-)$ in~\cite{Per22} 
is compatible with semiorthogonal decompositions, 
we have~$\Jac(\cD) \cong \Jac(\cD_1) \times \Jac(\cD_2)$.
\end{proof}

\begin{remark}
If the Euler pairing on~$\rK^\tp_1(\cD)$, $\rK^\tp_1(\cD_1)$, and~$\rK^\tp_1(\cD_2)$ 
is unimodular, skew-symmetric and positive definite, and therefore
endows~$\Jac(\cD_1)$, $\Jac(\cD_2)$, and~$\Jac(\cD)$ with principal polarizations,
then the isomorphism of Proposition~\ref{prop:ij-gluing} 
is compatible with these principal polarizations.
\end{remark}

\subsection{Serre functor of the gluing}

Here we give a general description of the Serre functor of the gluing of two categories
in terms of their Serre functors and the gluing bimodule.
So, let~$\cD = \langle \cD_1, \cD_2\rangle$ be a gluing with gluing bimodule~$\rG$.
We assume that the categories~$\cD_1$ and~$\cD_2$ are smooth and proper over~$\kk$
and~$\rG$ is perfect over~$\kk$, so that~$\cD$ is smooth and proper over~$\kk$ by Proposition~\ref{prop:gluing-sp}.
These assumptions imply that the categories~$\cD$, $\cD_1$, and~$\cD_2$ are saturated 
(i.e., all functors from these categories to~$\Db(\kk)$ are representable), see, e.g., \cite[Lemma~3.8]{KS-hfd};
in particular, they have Serre functors~$\bS$, $\bS_1$, and~$\bS_2$, respectively.
Moreover, the bimodule~$\rG$ determines an adjoint pair of functors defined as follows:
\begin{align*}
\rG^{12} &\colon \cD_2 \to \cD_1,
&& 
\RHom_{\cD_1}(\cF_1, \rG^{12}(\cF_2)) \coloneqq \rG(\cF_1,\cF_2),
\\
\rG^{21} &\colon \cD_1 \to \cD_2,
&& 
\RHom_{\cD_2}(\rG^{21}(\cF_1), \cF_2) \coloneqq \rG(\cF_1,\cF_2)
\end{align*}
(the adjunction~$\rG^{21} \vdash \rG^{12}$ is obvious from the definition).

\begin{lemma}
There is a canonical morphism of functors~$\zeta \colon \rG^{12} \circ \bS_2 \circ \rG^{21} \to \bS_1$.
\end{lemma}

\begin{proof}
Let~$\cF_1,\cF'_1 \in \cD_1$. 
Then there is a chain of bifunctorial isomorphisms
\begin{align*}
\Hom_{\cD_1}(\rG^{12}(\bS_2(\rG^{21}(\cF_1))), \bS_1(\cF'_1)) &\cong
\Hom_{\cD_1}(\cF'_1, \rG^{12}(\bS_2(\rG^{21}(\cF_1))))^\vee \\ &\cong
\Hom_{\cD_2}(\rG^{21}(\cF'_1), \bS_2(\rG^{21}(\cF_1)))^\vee \cong
\Hom_{\cD_2}(\rG^{21}(\cF_1), \rG^{21}(\cF'_1)).
\end{align*}
If~$\cF'_1 = \cF_1$, the right-hand side contains the canonical element~$\rG^{21}(\id_{\cF_1})$;
the corresponding element of the left-hand side provides the required morphism of functors.
\end{proof}

\begin{remark}
As~$\rG^{21}$ is left adjoint to~$\rG^{12}$,
it follows that~$\bS_2 \circ \rG^{21} \circ \bS_1^{-1}$ is right adjoint to~$\rG^{12}$
and the morphism~$\zeta$ is obtained 
from the adjunction counit~$\rG^{12} \circ (\bS_2 \circ \rG^{21} \circ \bS_1^{-1}) \to \id_{\cD_1}$
by composition with~$\bS_1$.
\end{remark}

For each~$\phi \in \rH^0(\rG(\cF_1,\cF_2))$ we denote by
\begin{equation*}
\phi_1 \colon \cF_1 \to  \rG^{12}(\cF_2)
\qquad\text{and}\qquad 
\phi_2 \colon \rG^{21}(\cF_1) \to \cF_2
\end{equation*}
the corresponding morphisms.

\begin{theorem}
\label{thm:serre-general}
Let~$\cD = \langle \cD_1, \cD_2 \rangle$ be a gluing of two smooth and proper dg-enhanced triangulated categories
with perfect gluing bimodule~$\rG$.
For~$\cF = (\cF_1,\cF_2,\phi) \in \cD$ consider the objects
\begin{align}
\label{eq:bcf1}
\bar\cF_1 &\coloneqq \Cone \Big( \rG^{12}(\bS_2(\rG^{21}(\cF_1))) 
\xrightarrow{\ \zeta_{\cF_1} \oplus \rG^{12}(\bS_2(\phi_2))\ } 
\bS_1(\cF_1) \oplus \rG^{12}(\bS_2(\cF_2)) \Big),
\\
\label{eq:bcf2}
\bar\cF_2 &\coloneqq \Cone \Big( \bS_2(\rG^{21}(\cF_1)) \xrightarrow{\ \bS_2(\phi_2)\ } \bS_2(\cF_2) \Big),
\end{align}
and the element~$\bar\phi \in \rH^0(\rG(\bar\cF_1,\bar\cF_2))$ 
corresponding to the morphism~$(\id, \pr_2) \colon \bar\cF_1 \to \rG^{12}(\bar\cF_2)$.
Then
\begin{equation*}
\bS(\cF_1,\cF_2,\phi) \cong (\bar\cF_1, \bar\cF_2, \bar\phi).
\end{equation*}
\end{theorem}

\begin{proof}
Set~$\bcF \coloneqq (\bcF_1,\bcF_2,\bar\phi) \in \cD$.
For~$\cG = (\cG_1,\cG_2,\psi) \in \cD$ we will construct an isomorphism functorial in~$\cG$
\begin{equation}
\label{eq:barcf-serre}
\RHom(\cG, \bar\cF) \cong \RHom(\cF, \cG)^\vee.
\end{equation}
This will prove that~$\bar\cF$ represents the functor~$\Hom(\cF,-)^\vee$, hence it is isomorphic to~$\bS(\cF)$.

First, assume~$\cG_1 = 0$, so that~$\cG = \cG_2[-1] \in \cD_2 = {}^\perp\cD_1$.
By Lemma~\ref{lem:hom-gluing} we have an isomorphism~$\RHom(\cG, \bar\cF) \cong \RHom(\cG_2, \bar\cF_2)$,
hence the defining triangle~\eqref{eq:bcf2} of~$\bcF_2$ gives a distinguished triangle
\begin{equation*}
\RHom(\cG_2, \bS_2(\rG^{21}(\cF_1))) \xrightarrow{\ \bS_2(\phi_2)\ }
\RHom(\cG_2,\bS_2(\cF_2)) \xrightarrow{\qquad}
\RHom(\cG, \bar\cF).
\end{equation*}
Dualizing this triangle and using Serre duality in~$\cD_2$, we obtain a distinguished triangle
\begin{equation*}
\RHom(\cG, \bar\cF)^\vee \xrightarrow{\qquad}
\RHom(\cF_2,\cG_2) \xrightarrow{\ \ \phi_2\ \ }
\RHom(\rG^{21}(\cF_1), \cG_2).
\end{equation*}
Comparing this with Lemma~\ref{lem:hom-gluing}, we obtain the required isomorphism~\eqref{eq:barcf-serre},
functorially in~$\cG_2$.

Next, assume that~$\psi_2 \colon \rG^{21}(\cG_1) \to \cG_2$ is an isomorphism, so that~$\cG \in {}^\perp\cD_2$.
Then by Lemma~\ref{lem:hom-gluing} we have~$\RHom(\cG, \bar\cF) \cong \RHom(\cG_1,\bar\cF_1)$,
and using the definition~\eqref{eq:bcf1} of~$\bar\cF_1$ we obtain a distinguished triangle
\begin{equation*}
\RHom(\cG_1, \rG^{12}(\bS_2(\rG^{21}(\cF_1)))) \xrightarrow{\ \zeta_{\cF_1} \oplus \rG^{12}(\bS_2(\phi_2))\ }
\RHom(\cG_1, \bS_1(\cF_1) \oplus \rG^{12}(\bS_2(\cF_2))) \xrightarrow{\qquad\qquad}
\RHom(\cG, \bar\cF).
\end{equation*}
Dualizing it and using adjunction of~$\rG^{21}$ and~$\rG^{12}$ and Serre duality in~$\cD_1$ and~$\cD_2$, we obtain
\begin{equation*}
\RHom(\cG, \bar\cF)^\vee \to
\RHom(\cF_1, \cG_1) \oplus \RHom(\cF_2, \rG^{21}(\cG_1)) \xrightarrow{\ (\rG^{21},\phi_2)\ }
\RHom(\rG^{21}(\cF_1), \rG^{21}(\cG_1)),
\end{equation*}
and comparing this with Lemma~\ref{lem:hom-gluing}, we obtain the required isomorphism~\eqref{eq:barcf-serre},
functorially in~$\cG_1$.
 
Now, for an arbitrary object~$\cG = (\cG_1,\cG_2,\psi)$, we consider its decomposition triangle 
with respect to the semiorthogonal decomposition~$\cD = \langle \cD_2, {}^\perp\cD_2 \rangle$
that has the following form
\begin{equation*}
(0,\cG'_2,0) \xrightarrow{\qquad}
(\cG_1, \rG^{21}(\cG_1), \id) \xrightarrow{\ (\id_{\cG_1}, \psi_2)\ }
(\cG_1,\cG_2,\psi),
\end{equation*}
where~$\cG'_2 \coloneqq \Cone \big(\rG^{21}(\cG_1) \xrightarrow{\ \psi_2\ } \cG_2 \big)[-1]$.
Above we established isomorphisms~\eqref{eq:barcf-serre} for the first and second vertices of this triangle.
It is also easy to see that they are functorial with respect to the first arrow of the above triangle;
therefore, it follows that~\eqref{eq:barcf-serre} also holds for the third vertex, as required.
\end{proof}

\begin{remark}
A similar computation gives a formula for the inverse Serre functor.
Namely, let 
\begin{align*}
\tilde\cF_1 &\coloneqq \Cone\Big( 
\bS_1^{-1}(\cF_1) \xrightarrow{\ \bS_1^{-1}(\phi_1)\ } 
\bS_1^{-1}(\rG^{12}(\cF_2)) 
\Big)[-1],
\\
\tilde\cF_2 &\coloneqq \Cone\Big( 
\rG^{21}(\bS_1^{-1}(\cF_1)) \oplus \bS_2^{-1}(\cF_2) \xrightarrow{\ \rG^{21}(\bS_1^{-1}(\phi_1)) \oplus \xi_{\cF_2}\ } 
\rG^{21}(\bS_1^{-1}(\rG^{12}(\cF_2))) 
\Big)[-1],
\end{align*}
where~$\xi \colon \bS_2^{-1} \to \rG^{21} \circ \bS_1^{-1} \circ \rG^{12}$ 
is the morphism of functors defined analogously to~$\zeta$.
Then
\begin{equation*}
\bS^{-1}(\cF_1,\cF_2,\phi) = (\tilde\cF_1,\tilde\cF_2,\tilde\phi), 
\end{equation*}
where~$\tilde\phi$ corresponds to the morphism~$(\mathrm{in}_1,\id) \colon \rG^{21}(\tilde\cF_1) \to \tilde\cF_2$,
and~$\mathrm{in}_1$ is the first embedding.
\end{remark}

\section{Augmented curves}
\label{sec:ac}

In this section we study the first example of a gluing --- 
the gluing of the derived category of a point and the derived category of a curve,
where the gluing object~$\urG$ is the structure sheaf of the curve.

\begin{definition}
{\sf The augmented curve} (or {\sf augmentation of})~$C$
is the triangulated category~$\Db(\cO,C)$ that admits a semiorthogonal decomposition
\begin{equation*}
\Db(\cO,C) = \langle \cE, \Db(C) \rangle,
\end{equation*}
where~$C$ is a smooth projective curve and~$\cE$ is an exceptional object such that
\begin{equation*}
\Ext_{\Db(\cO,C)}^\bullet(\cE, \cF) \cong \rH^\bullet(C, \cF)
\qquad\text{for any~$\cF \in \Db(C)$}.
\end{equation*}
In other words, the category~$\Db(\cO,C)$ is the gluing of~$\Db(\kk)$ and~$\Db(C)$ 
with the gluing bimodule~$\rG$ defined by~$\rG(V,\cF) \coloneqq \RHom_{\Db(C)}(V \otimes \cO_C, \cF)$
for~$V \in \Db(\kk)$ and~$\cF \in \Db(C)$.

The object~$\cE$ as above is called {\sf the canonical exceptional object} of~$\Db(\cO,C)$.
\end{definition}

\begin{remark}
\label{rem:ac-lb}
If instead of~$\urG = \cO_C$ we take~$\urG$ to be another line bundle on~$C$,
the result of the gluing of~$\Db(\kk)$ and~$\Db(C)$ will stay the same (up to equivalence), see Corollary~\ref{cor:gluing-auto}.
It will also stay the same if we swap the factors, i.e., consider the gluing of~$\Db(C)$ and~$\Db(\kk)$, see Remark~\ref{rem:gluing-dual}.
\end{remark}

One could also consider the gluing of~$\Db(\kk)$ and~$\Db(C)$ 
using the structure sheaf of a point~$\cO_x \in \Db(C)$ as the gluing object.
The next lemma shows that the result is a familiar category.

\begin{lemma}
\label{lem:point-augmentation}
Let~$x \in C(\kk)$ be a $\kk$-point of a curve~$C$.
The gluing of~$\Db(\kk)$ and~$\Db(C)$ with the gluing object~$\urG = \cO_x$ 
is equivalent to the derived category~$\Db(\sqrt{C,x})$ of the square root stack~$\sqrt{C,x}$.
\end{lemma}

\begin{proof}
By~\cite[Theorem~1.6]{IU15} the derived category of the square root stack has a semiorthogonal decomposition
\begin{equation*}
\Db(\sqrt{C,x}) = \langle \Db(\kk), \Db(C) \rangle,
\end{equation*}
with the first component generated by the exceptional object~$\cO_x \otimes \chi$
(where~$\chi$ is the nontrivial character of the isotropy group~$\upmu_2$ of the stacky point~$x$)
and the second component embedded by the pullback functor with respect to the natural morphism~$\sqrt{C,x} \to C$.
Furthermore, the argument in the proof of~\cite[Theorem~1.6]{IU15} shows that
\begin{equation*}
\RHom_{\Db(\sqrt{C,x})}(\cO_x \otimes \chi, \cF) \cong 
\RHom_{\Db(x)}(\cO_x, \cF\vert_x)[-1] \cong
\rH^\bullet(C, \cF \otimes \cO_x[-1]),
\end{equation*}
where~$\cF\vert_x$ stands for the derived restriction of~$\cF$ to~$x$.
This means that the gluing object is isomorphic to~$\cO_x[-1]$.
Applying Corollary~\ref{cor:gluing-auto}, we see that the gluing with~$\urG \cong \cO_x$ is also equivalent to~$\Db(\sqrt{C,x})$.
\end{proof}

\subsection{Basic properties and Serre functor}

It is not hard to see that any augmentation~$\Db(\cO,C)$ 
can be realized as an admissible subcategory of the derived category of a smooth projective variety.

\begin{lemma}
\label{lem:ac-blowup}
Let~$C \hookrightarrow X$ be an embedding of~$C$ into a smooth projective variety~$X$ of dimension at least~$3$
such that the sheaf~$\cO_X$ is exceptional.
Let~$\pi \colon \tX \coloneqq \Bl_C(X) \to X$ be the blowup,
let~$i \colon E \to \tX$ be the embedding of the exceptional divisor, 
and let~$p \colon E \to C$ be the natural projection.
Then the subcategory
\begin{equation*}
\langle \cO_{\tX}, i_*p^*\Db(C) \rangle \subset \Db(\tX)
\end{equation*}
is admissible and equivalent to the augmentation~$\Db(\cO,C)$.
\end{lemma}

\begin{proof}
The blowup formula implies that the functor~$i_* \circ p^* \colon \Db(C) \to \Db(\tX)$ is fully faithful
and the object~$\cO_{\tX}$ is exceptional and semiorthogonal to~$i_*p^*\Db(C)$;
therefore, the subcategory generated by~$\cO_{\tX}$ and~$i_*p^*\Db(C)$ is admissible in~$\Db(\tX)$.
It remains to note that for all~$\cF \in \Db(C)$ we have
\begin{equation*}
\RHom_{\Db(\tX)}(\cO_\tX, i_*p^*\cF) \cong
\RHom_{\Db(E)}(\cO_E, p^*\cF) \cong
\RHom_{\Db(E)}(p^*\cO_C, p^*\cF) \cong
\RHom_{\Db(C)}(\cO_C, \cF),
\end{equation*}
hence we have an equivalence~$\Db(\cO,C) \simeq \langle \cO_{\tX}, i_*p^*\Db(C) \rangle \subset \Db(\tX)$.
\end{proof}

Now we apply Propositions~\ref{prop:gluing-sp}, \ref{prop:invariants-general}, and~\ref{prop:ij-gluing} 
to compute the basic invariants of~$\Db(\cO,C)$.

\begin{proposition}
\label{prop:invariants-ac}
Let~$C$ be a smooth proper curve of genus~$g$.
The augmentation~$\Db(\cO,C)$ is a smooth and proper triangulated category.
Moreover, we have
\begin{arenumerate}
\item 
\label{it:additive-ac}
$\HH_\bullet(\Db(\cO,C)) = \kk^g[1] \oplus \kk^3 \oplus \kk^g[-1]$
and
$\rK_0(\Db(\cO,C)) = \Z^3 \oplus \Pic^0(C)$.
\item 
\label{it:rkn-ac}
$\rKn(\Db(\cO,C)) = \Z^3$; if, furthermore, $C(\kk) \ne \varnothing$ and~$x \in C(\kk)$,
then the matrix of the Euler form in the basis~$[\cE]$, $[\cO_C]$, $[\cO_{x}]$
\textup(where we use the embedding functor~$\Db(C) \hookrightarrow \Db(\cO,C)$
to realize~$\cO_C$ and~$\cO_x$ as objects of~$\Db(\cO,C)$\textup)
is
\begin{equation*}
\upchi_{\Db(\cO,C)} = \begin{pmatrix} 1 & 1 - g & 1 \\ 0 & 1 - g & 1 \\ 0 & -1 & 0 \end{pmatrix}.
\end{equation*}
\item 
\label{it:jac-ac}
If~$\kk = \CC$, the Euler pairing on~$\rK^\tp_1(\Db(\cO,C))$ is unimodular, skew-symmetric, and positive definite,
and~$\Jac(\Db(\cO,C)) \cong \Jac(C)$ is an isomorphism of principally polarized abelian varieties.
\item 
\label{it:hhc-ac}
If~$g \ge 2$, then~$\HH^\bullet(\Db(\cO,C)) = \kk \oplus \kk^{3g-3}[-2]$.
\end{arenumerate}
\end{proposition}

\begin{proof}
The smoothness and properness of~$\Db(\cO,C)$ follow from Proposition~\ref{prop:gluing-sp}.

Parts~\ref{it:additive-ac} and~\ref{it:rkn-ac} of the proposition
follow immediately from Proposition~\ref{prop:invariants-general}.

To prove part~\ref{it:jac-ac}, we apply Remark~\ref{rem:ppav} taking into account the vanishing~$\rK^\tp_1(\Db(\CC)) = 0$.
As a result we obtain an isometry~$\rK^\tp_1(\Db(\cO,C)) \cong \rK^\tp_1(\Db(C))$,
and since the Euler form on the right side is unimodular, skew-symmetric, and positive definite (by~\cite[Proposition~5.23]{Per22}),
we conclude the same is true for the left side and obtain an isomorphism of the intermediate Jacobians.

To prove part~\ref{it:hhc-ac}, we note that for~$g \ge 2$ the Hochschild--Kostant--Rosenberg isomorphism gives
\begin{equation*}
\HH^\bullet(\Db(C)) = 
\rH^0(C,\cO_C) \oplus \rH^1(C,\cO_C)[-1] \oplus \rH^1(C, \cT_C)[-2] =
\kk \oplus \kk^g[-1] \oplus \kk^{3g - 3}[-2].
\end{equation*}
On the other hand, $\Ext^\bullet(\urG, \urG) = \rH^0(C, \cO_C) \oplus \rH^1(C, \cO_C)[-1] = \kk \oplus \kk^g[-1]$.
It remains to note that the morphism~$\rH^1(C,\cO_C) = \HH^1(\Db(C)) \to \Ext^1(\urG, \urG) = \rH^1(C,\cO_C)$
from Proposition~\ref{prop:invariants-general}\ref{it:hh-triangle}
is induced by tensor product with~$\urG \cong \cO_C$,
hence it is an isomorphism.
\end{proof}

\begin{remark}
\label{rem:hhc-ac-g01}
If~$g = 1$, a similar computation gives~$\HH^\bullet(\Db(\cO,C)) = \kk \oplus \kk[-1] \oplus \kk[-2]$,
and if~$g = 0$, then~$\HH^\bullet(\Db(\cO,C)) = \kk \oplus \kk[-1]^{\oplus 3}$.
\end{remark}

\begin{remark}
It follows from Proposition~\ref{prop:invariants-ac}\ref{it:hhc-ac} and Remark~\ref{rem:hhc-ac-g01} that
\begin{equation*}
\HH^2(\Db(\cO,C)) = \HH^2(\Db(C)) = \rH^1(C, \cT_{C}).
\end{equation*}
This means that any (infinitesimal) deformation of the augmentation~$\Db(\cO,C)$ 
is induced by a deformation of the curve~$C$.
\end{remark}

It also follows that the tangent space~$\HH^1(\Db(\cO,C))$ 
to the group of autoequivalences~$\Aut(\Db(\cO,C))$ of an augmented curve of genus~$g \ge 2$ is trivial, 
hence the group is discrete (see~\cite[Corollary~0.3]{TV}).
Furthermore, Theorem~\ref{thm:gluing} implies that~$\Aut(\Db(\cO,C))$ contains the following obvious subgroup:
\begin{equation*}
\ZZ[1] {}\times{} \Aut(C) \subset \Aut(\Db(\cO,C)).
\end{equation*}
In the case where~$\g(C) = 1$, yet another autoequivalence of~$\Db(\cO,C)$ 
is induced by the spherical twist~$\bT_{\cO_C}$ (which fixes~$\cO_C$).
We will also see extra autoequivalences for augmented curves of genera~$g = 3$ and~$g = 4$ 
(Lemmas~\ref{lem:ac-3} and~\ref{lem:ac-4}).
It will be interesting to describe the group~$\Aut(\Db(\cO,C))$ in general.

Another interesting autoequivalence of~$\Db(\cO,C)$ that exists for any curve~$C$ is the Serre functor. 
To describe it, we apply Theorem~\ref{thm:serre-general}.
Recall Notation~\ref{not:fff}; according to it, the objects of~$\Db(\cO,C)$ are triples~$(V,\cF,\phi)$,
consisting of a graded vector space, a complex of coherent sheaves on~$C$, and a morphism~$\phi \colon V \otimes \cO_C \to \cF$.

\begin{theorem}
\label{thm:serre-ac}
For~$(V,\cF,\phi) \in \Db(\cO,C)$ consider the objects
\begin{align*}
\bar{V} &\coloneqq \Cone \Big( V \otimes \rH^0(C,\omega_C)
\xrightarrow{\ \phi \otimes \ev_{\omega_C}\ } 
\rH^\bullet(C, \cF \otimes \omega_C) \Big)[1],
\\
\bar\cF &\coloneqq \Cone \Big( V \otimes \omega_C 
\xrightarrow{\ \phi \otimes \omega_C\ } 
\cF \otimes \omega_C \Big)[1],
\end{align*}
and the element~$\bar\phi \in \Hom(\bar{V} \otimes \cO_C, \bar\cF)$ 
induced by the evaluation morphisms of~$\omega_C$ and~$\cF \otimes \omega_C$.
Then
\begin{equation*}
\bS_{\Db(\cO,C)}(V,\cF,\phi) \cong (\bar{V}, \bar\cF, \bar\phi).
\end{equation*}
\end{theorem}

\begin{proof}
We just apply Theorem~\ref{thm:serre-general}. 
Since~$\rG^{12}(\cF) = \rH^\bullet(C,\cF)$ and~$\rG^{21}(V) = V \otimes \cO_C$, we have
\begin{align*}
\bar{V} &= \Cone \Big( V \otimes \rH^\bullet(C,\omega_C[1])
\xrightarrow{\ \zeta_{V} \oplus \rH^\bullet(C,\phi \otimes \omega_C[1])\ } 
V \oplus \rH^\bullet(C,(\cF \otimes \omega_C[1])) \Big),
\\
\bar\cF &= \Cone \Big( V \otimes \omega_C[1] \xrightarrow{\ \phi \otimes \omega_C \ } \cF \otimes \omega_C[1]  \Big).
\end{align*}
It remains to note that in the formula for~$\bar{V}$ 
the map~$\zeta_V$ induces an isomorphism of~$V \otimes \rH^0(C,\omega_C[1]) = V$
with the first summand in~$V \oplus \rH^\bullet(C,(\cF \otimes \omega_C[1]))$.
\end{proof}

\begin{remark}
Assume~$x \in C(\kk) \ne \varnothing$.
Using Theorem~\ref{thm:serre-ac}, it is easy to compute the automorphism of~$\rKn(\Db(\cO,C))$ induced by the Serre functor.
In the basis~$[\cE]$, $[\cO_C]$, $[\cO_{x}]$ it is given by the matrix
\begin{equation*}
\begin{pmatrix}
g & g-1 & 1 \\
-1 & -1 & 0 \\
2-2g & 2-2g & -1
\end{pmatrix},
\end{equation*}
whose characteristic polynomial is~$(t-1)(t^2 - (g-3)t + 1)$.
In particular, it is quasiunipotent if and only if~$1 \le g \le 5$ and unipotent if and only if~$g = 5$.
In particular, for~$g \ne 5$ the category~$\Db(\cO,C)$ is not equivalent to the derived category of a variety.
Later we will see that for~$g = 5$ and~$C$ general the category~$\Db(\cO,C)$ 
is a twisted derived category of a Deligne--Mumford stack, see Proposition~\ref{prop:ac-g5}.
\end{remark}

Objects on which (a power of) the Serre functor acts as a shift play an important role
(examples of such objects are the structure sheaves of points in the derived categories of varieties,
or spherical objects in arbitrary categories).
The following lemma gives a way to construct such objects in~$\Db(\cO,C)$.

We will use the following

\begin{definition}
\label{def:augmented-sheaf}
If~$\cF$ is a sheaf on~$C$, the object
\begin{equation}
\label{eq:augmentation}
\fa(\cF) \coloneqq (\rH^0(C,\cF), \cF, \ev_\cF) \in \Db(\cO,C)
\end{equation}
(where~$\ev_\cF \colon \rH^0(C,\cF) \otimes \cO_C \to \cF$ is the evaluation morphism)
is called {\sf the augmentation} of~$\cF$.
\end{definition}

\begin{lemma}
\label{lem:serre-lower-bound}
Let~$(\cF_1,\cF_2)$ be globally generated vector bundles on~$C$ such that~$\rH^1(C,\cF_i \otimes \omega_C) = 0$,
the multiplication maps~$\rH^0(C,\cF_i) \otimes \rH^0(C,\omega_C) \to \rH^0(C, \cF_i \otimes \omega_C)$ are surjective, and 
\begin{align}
\label{eq:serre-pair-cf2}
\cF_2 &\cong \Ker \Big(\rH^0(C,\cF_1) \otimes \omega_C \xrightarrow{\ \ev_{\cF_1} \otimes \omega_C\ } \cF_1 \otimes \omega_C \Big),
\\
\label{eq:serre-pair-cf1}
\cF_1 &\cong \Ker \Big(\rH^0(C,\cF_2) \otimes \omega_C \xrightarrow{\ \ev_{\cF_2} \otimes \omega_C\ } \cF_2 \otimes \omega_C \Big).
\end{align}
Then the augmentations~$\fa(\cF_i) = (\rH^0(C, \cF_i),\cF_i,\ev_{\cF_i})$ are swapped by the Serre functor up to shifts:
\begin{equation}
\label{eq:bs-fa-cfi}
\bS_{\Db(\cO,C)}(\fa(\cF_1)) \cong \fa(\cF_2)[2],
\qquad 
\bS_{\Db(\cO,C)}(\fa(\cF_2)) \cong \fa(\cF_1)[2],
\end{equation}
and in particular~$\bS^2_{\Db(\cO,C)}(\fa(\cF_i)) \cong \fa(\cF_i)[4]$.
\end{lemma}

\begin{proof}
Let~$V_i \coloneqq \rH^0(C,\cF_i)$, so that~$\fa(\cF_i) = (V_i,\cF_i,\ev_{\cF_i})$.

Comparing the assumptions~\eqref{eq:serre-pair-cf2}, \eqref{eq:serre-pair-cf1} with Theorem~\ref{thm:serre-ac},
and using the global generation of~$\cF_i$,
we see that~$\bar\cF_1 \cong \cF_2[2]$ and~$\bar\cF_2 {}\cong{} \cF_1[2]$.
Furthermore, \eqref{eq:serre-pair-cf2} gives the left exact sequence
\begin{equation*}
0 \to 
\rH^0(C,\cF_2) \to 
\rH^0(C,\cF_1) \otimes \rH^0(C,\omega_C) \to 
\rH^0(C, \cF_1 \otimes \omega_C),
\end{equation*}
whose second arrow is the multiplication map.
Since this map is surjective and~$\rH^1(C,\cF_1 \otimes \omega_C) = 0$ by assumption, 
it follows that~$\overline{V_1} \cong V_2[2]$.
A similar argument shows that~$\overline{V_2} \cong V_1[2]$, 
and~\eqref{eq:bs-fa-cfi} follows.
\end{proof}

\subsection{BN-exceptional objects}
\label{ss:bn-augmentations}

In this subsection we show that Brill--Noether--Petri extremal line bundles on~$C$
give rise to exotic exceptional objects in~$\Db(\cO,C)$ and study their orthogonal complements.

We will say that a line bundle~$\cL$ on a curve~$C$ is {\sf Brill--Noether--Petri} ({\sf BNP}) {\sf extremal}
if the Petri map 
\begin{equation*}
\rH^0(C, \cL) \otimes \rH^0(C, \cL^\vee(K_C)) \to \rH^0(C, \cO_C(K_C))
\end{equation*}
is an isomorphism (and therefore~$h^0(\cL)\cdot h^1(\cL) = \g(C)$).
Note that if~$C$ is general and~$\kk$ is algebraically closed of characteristic zero,
then for any factorization~$\g(C) = r \cdot s$ there are~$\deg(\Gr(r,r+s))$ BNP extremal line bundles on~$C$,
and each of these line bundles is globally generated, see~\cite[Lemma~2.3 and Remark~2.4]{BKM}.

\begin{proposition}
\label{prop:bnp-exceptional}
If~$\cL$ is a BNP extremal line bundle on a curve~$C$, then the augmentation
\begin{equation*}
\cE_\cL \coloneqq \fa(\cL) = (\rH^0(C, \cL), \cL, \ev_\cL)
\end{equation*}
is an exceptional object in~$\Db(\cO,C)$.
\end{proposition}

\begin{proof}
The cohomology exact sequence of Lemma~\ref{lem:hom-gluing} computing~$\Ext_{\Db(\cO,C)}^\bullet(\cE_\cL, \cE_\cL)$ looks like
\begin{alignat*}{7}
0 &\to 
\Hom_{\Db(\cO,C)}(\cE_\cL, \cE_\cL) &&\to 
\Hom_C(\cL, \cL) \oplus \Hom(\rH^0(C, \cL), \rH^0(C, \cL))  &&\to 
\Hom_C(\rH^0(C, \cL) \otimes \cO_C, \cL) \\ &\to
\Ext_{\Db(\cO,C)}^1(\cE_\cL, \cE_\cL) &&\to 
\Ext_C^1(\cL, \cL) &&\to 
\Ext^1_C(\rH^0(C, \cL) \otimes \cO_C, \cL) \to
\dots.
\end{alignat*}
The map~$\Hom(\rH^0(C, \cL), \rH^0(C, \cL)) \to \Hom_C(\rH^0(C, \cL) \otimes \cO_C, \cL)$ in the first row
is obviously an isomorphism, hence~$\Hom_{\Db(\cO,C)}(\cE_\cL, \cE_\cL) \cong \Hom_C(\cL, \cL) \cong \kk$.
Further, the map 
\begin{equation*}
\Ext_C^1(\cL, \cL) \to 
\Ext^1_C(\rH^0(C, \cL) \otimes \cO_C, \cL) \cong
\rH^0(C, \cL)^\vee \otimes \rH^1(C, \cL)
\end{equation*}
in the second row is dual to the Petri map, hence it is an isomorphism as well, 
and we conclude that~$\Ext_{\Db(\cO,C)}^i(\cE_\cL, \cE_\cL) = 0$ for all~$i \ne 0$,
hence~$\cE_\cL$ is exceptional.
\end{proof}

\begin{definition}
\label{def:bn-exceptional}
If~$\cL$ is a BNP extremal line bundle on a curve~$C$,
the exceptional object~$\cE_\cL = \fa(\cL)$ defined in Proposition~\ref{prop:bnp-exceptional}
is called the {\sf Brill--Noether (BN) exceptional object}.
\end{definition}

\begin{remark}
\label{rem:bnex-o-om}
Using Theorem~\ref{thm:serre-ac}, it is easy to check that
\begin{equation*}
\cE_{\cO_C} \cong \bS^{-1}(\cE)
\qquad\text{and}\qquad
\cE_{\omega_C} \cong \bS(\cE)[-2].
\end{equation*}
Thus, these two BN-exceptional objects can be obtained 
from the canonical exceptional object~$\cE$ by an autoequivalence of~$\Db(\cO,C)$.
In particular, the orthogonal complements~${}^\perp\cE_{\cO_C}$ and~${}^\perp\cE_{\omega_C}$ 
of these objects in~$\Db(\cO,C)$ are equivalent to~$\Db(C)$.
\end{remark}

The orthogonal complements of other BN-exceptional objects turn out to be interesting.

\begin{definition}
\label{def:bn-modification}
If~$\cE_\cL$ is a BN-exceptional object in~$\Db(\cO,C)$, the category~$\cE_\cL^\perp \simeq {}^\perp\cE_\cL \subset \Db(\cO,C)$
is called the {\sf BN-modification} of~$\Db(C)$ with respect to~$\cL$,
or simply {\sf $\cL$-modification}.
\end{definition}

A geometrically meaningful example of a BN-modification is given in Appendix~\ref{sec:cubic-3}.

\begin{proposition}
\label{prop:invariants-mac}
Let~$C$ be a smooth curve of genus~$g$ with a BNP extremal line bundle~$\cL$.
The $\cL$-modification~${}^\perp\cE_\cL$ of~$\Db(C)$ is a smooth and proper triangulated category.
Moreover, we have
\begin{arenumerate}
\item 
\label{it:additive-mac}
$\HH_\bullet({}^\perp\cE_\cL) = \kk^g[1] \oplus \kk^2 \oplus \kk^g[-1]$
and
$\rK_0({}^\perp\cE_\cL) = \Z^2 \oplus \Pic^0(C)$.
\item 
\label{it:rkn-mac}
$\rKn({}^\perp\cE_\cL) = \Z^2$; if, furthermore, $C(\kk) \ne \varnothing$ and~$x \in C(\kk)$,
then the matrix of the Euler form in the basis~$[\cE] - h^1(\cL)[\cO_{x}]$, $[\cE] - [\cO_C] - h^0(\cL)[\cO_{x}]$ is
\begin{equation*}
\upchi = 
\begin{pmatrix}
1 - h^1(\cL) & \g(C) - h^0(\cL) - h^1(\cL) \\
1 & 1 - h^0(\cL)
\end{pmatrix}.
\end{equation*}
\item 
\label{it:jac-mac}
If~$\kk = \CC$, then
$\Jac({}^\perp\cE_\cL) \cong \Jac(C)$ is an isomorphism of principally polarized abelian varieties.
\end{arenumerate}
\end{proposition}

\begin{proof}
The smoothness and properness of~${}^\perp\cE_\cL$
follow from Propositions~\ref{prop:gluing-sp} and~\ref{prop:invariants-ac}.

Part~\ref{it:additive-mac} follows from additivity of Hochschild homology and~$\rK_0$
and Proposition~\ref{prop:invariants-ac}\ref{it:additive-ac}.

Part~\ref{it:rkn-mac} follows easily from Proposition~\ref{prop:invariants-ac}\ref{it:rkn-ac}.
Indeed, since~$[\cE_\cL] = h^0(\cL)[\cE] - [\cL]$,
we have
\begin{equation*}
\upchi([\cE],[\cE_\cL]) = h^0(\cL) - \upchi(\cL) = h^1(\cL),\qquad
\upchi([\cO_C],[\cE_\cL]) = -\upchi(\cL),\qquad
\upchi([\cO_{\pt}],[\cE_\cL]) = 1,
\end{equation*}
hence the elements~$[\cE] - h^1(\cL)[\cO_{x}]$ and~$[\cE] - [\cO_C] - h^0(\cL)[\cO_{x}]$ 
generate~$\rKn({}^\perp\cE_\cL)$,
and a straightforward computation gives the matrix of the Euler form in this basis.

Part~\ref{it:jac-mac} follows from Proposition~\ref{prop:invariants-ac}\ref{it:jac-ac} and Remark~\ref{rem:ppav}.
\end{proof}

In particular, Proposition~\ref{prop:invariants-mac} shows that~${}^\perp\cE_\cL$ 
has the same Hochschild homology, Grothendieck group, and intermediate Jacobian as the derived category of a curve,
so the next result may look surprising.

\begin{corollary}
\label{cor:bnex-ort}
Let~$\cL$ be a BNP extremal line bundle on a curve~$C$.
If~$\cL \not\cong \cO_C$ and~$\cL \not\cong \omega_C$,
then the BN-modification~${}^\perp\cE_\cL$ of~$\Db(C)$ is not equivalent to the derived category of a curve.
\end{corollary}

\begin{proof}
Assume~${}^\perp\cE_\cL \simeq \Db(C')$.
Then there is an isometry~$\rKn({}^\perp\cE_\cL) \cong \rKn(\Db(C'))$ of lattices
with respect to the symmetrizations of the Euler forms.
Passing to an appropriate field extension, we may assume that both~$C$ and~$C'$ have $\kk$-points.
Then by Proposition~\ref{prop:invariants-mac}\ref{it:rkn-mac}
the symmetrization~$\upchi_{\Sym}$ of the bilinear form~$\upchi$ of~${}^\perp\cE_\cL$ takes the form~$-ax^2 +abxy - by^2$,
where~$a = h^0(\cL) - 1$ and~$b = h^1(\cL) - 1$,
while for~$\Db(C')$ it takes the form~$(1-\g(C'))r^2$.
Thus, a necessary condition for the equivalence~${}^\perp\cE_\cL \simeq \Db(C')$
is the vanishing of the discriminant~$a^2b^2 - 4ab = ab(ab - 4)$ of~$\upchi_{\Sym}$, 
which gives one of the following possibilities:
\begin{itemize}
\item 
$a = 0$ or~$b = 0$, or
\item 
$a = b = 2$ or~$\{a,b\} = \{1,4\}$.
\end{itemize}
The cases~$a = 0$ and~$b = 0$ correspond to~$\cL = \cO_C$ and~$\cL = \omega_C$, respectively,
so it remains to show that the other two cases are impossible.
Note that~$\g(C) = h^0(\cL)\cdot h^1(\cL) = (a + 1)(b + 1)$.

Indeed, if~$a = b = 2$, then~$\upchi_{\Sym} = -2(x-y)^2$, 
while the symmetrized Euler form of a curve of the corresponding genus~$g = (a+1)(b+1) = 9$ takes the form~$-8r^2$,
and these forms are not equivalent.

Similarly, if~$a = 4$ and~$b = 1$, then~$\upchi_{\Sym} = -(2x-y)^2$, 
while the symmetrized Euler form of a curve of the corresponding genus~$g = (a+1)(b+1) = 10$ takes the form~$-9r^2$,
and these forms are not equivalent either. 
The case where~$a = 1$ and~$b = 4$ is analogous.
\end{proof}

\begin{remark}
\label{rem:hhc-bnp-mod}
The computation of Hochschild cohomology of the BN-modification~${}^\perp\cE_\cL$ is much harder.
For completeness, we provide a sketch under a generality assumption.

Let~$\cL \not\in \{\cO_C, \omega_C\}$ be a globally generated BNP extremal line bundle
and assume the morphism
\begin{equation*}
\rH^0(C,\cL) \otimes \rH^0(C,\omega_C) \to \rH^0(\cL \otimes \omega_C)
\end{equation*}
is surjective.
Then~$\bS(\cE_\cL) = \fa(\cF_\cL \otimes \omega_C)[2]$, 
where~$\cF_\cL$ is the vector bundle defined from the exact sequence
\begin{equation*}
0 \to \cF_\cL \to \rH^0(C, \cL) \otimes \cO_C \to \cL \to 0.
\end{equation*}
Furthermore, denoting by~$\bar\cL \coloneqq \cL^{-1} \otimes \omega_C$ the adjoint BNP extremal line bundle,
it is not hard to check that the groups~$\Ext^p(\bS(\cE_\cL), \cE_\cL)$ with~$p \in \{2,3,4\}$ 
are isomorphic to the cohomology of the complex
\begin{equation}
\label{eq:l-bl-omega}
\rH^1(\omega_C^{-1}) \to
\Big(\rH^0(\cL)^\vee \otimes \rH^1(\bar\cL^{-1})\Big) \oplus
\Big(\rH^1(\cL^{-1}) \otimes \rH^0(\bar\cL)^\vee\Big) \to
\rH^0(\cL)^\vee \otimes \rH^1(\cO_C) \otimes \rH^0(\bar\cL)^\vee,
\end{equation}
where the maps are dual to the natural multiplication maps,
and the groups vanish for~$p \not\in \{2,3,4\}$.
Indeed, using the definition of~$\cF_\cL$ and the surjectivity assumption, we obtain an exact sequence
\begin{equation*}
0 \to \rH^0(\cF_\cL \otimes \omega_C) \to \rH^0(\cL) \otimes \rH^0(\omega_C) \to \rH^0(\cL \otimes \omega_C) \to 0.
\end{equation*}
Similarly, tensoring the defining sequence of~$\cF_\cL$ by~$\omega_C$ and applying the functor~$\Ext^\bullet(-,\cL)$,
we obtain
\begin{equation*}
0 \to \Hom(\cF_\cL \otimes \omega_C, \cL) \to
\rH^1(\omega_C^{-1}) \to
\rH^0(\cL)^\vee \otimes \rH^1(\bar\cL^{-1}) \to
\Ext^1(\cF_\cL \otimes \omega_C, \cL) \to 0.
\end{equation*}
Dualizing the first sequence to compute~$\rH^0(\cF_\cL \otimes \omega_C)^\vee$
and using Lemma~\ref{lem:hom-gluing}, we deduce that the groups~$\Ext^\bullet(\fa(\cF_\cL \otimes \omega_C), \cE_\cL)$
are isomorphic to the cohomology of the complex
\begin{equation*}
\rH^1(\omega_C^{-1}) \to
\rH^0(\cL)^\vee \otimes \rH^1(\bar\cL^{-1}) \to
\Big( \rH^0(\cL)^\vee \otimes \rH^0(\omega_C)^\vee \otimes \rH^1(\cL) \Big) \Big/ \Big( \rH^0(\cL \otimes \omega_C)^\vee \otimes \rH^1(\cL) \Big).
\end{equation*}
Finally, using Serre duality, we rewrite this complex as~\eqref{eq:l-bl-omega}.

On the other hand, by~\cite[Theorem~3.3 and Proposition~3.7]{K15} 
Hochschild cohomology of~${}^\perp\cE_\cL$ is computed as the cone of the morphism
\begin{equation*}
\Ext^\bullet(\bS(\cE_\cL), \cE_\cL) \to \HH^\bullet(\Db(\cO,C)) = \kk \oplus \rH^1(\omega_C^{-1})[-2].
\end{equation*}
It is not hard to see that the map~$\Ext^2(\bS(\cE_\cL), \cE_\cL) \to \rH^1(\omega_C^{-1})$ 
is the natural embedding of the kernel of the first arrow in~\eqref{eq:l-bl-omega}.
Therefore, it follows that~$\HH^0({}^\perp\cE_\cL) = \kk$, while
$\HH^2({}^\perp\cE_\cL)$ and~$\HH^3({}^\perp\cE_\cL)$ are isomorphic to the kernel and cokernel of the map
\begin{equation*}
\Big(\rH^0(\cL)^\vee \otimes \rH^1(\bar\cL^{-1})\Big) \oplus
\Big(\rH^1(\cL^{-1}) \otimes \rH^0(\bar\cL)^\vee\Big) \to
\rH^0(\cL)^\vee \otimes \rH^1(\cO_C) \otimes \rH^0(\bar\cL)^\vee, 
\end{equation*}
and all other Hochschild cohomology groups vanish.
In particular, it follows that the group of autoequivalences of a BN-modification as above is discrete,
again by~\cite[Corollary~0.3]{TV}.
\end{remark}

\subsection{Augmented curves of small genus}
\label{ss:ac579}

We finish this section by discussing the special features of some augmented curves.
Throughout this subsection we assume that~$C(\kk) \ne \varnothing$
and leave it to the interested readers to check what happens otherwise.

We start with the example of~$g = 0$.
The following description follows immediately from the definition.

\begin{lemma}
\label{lem:ac-0}
If~$\g(C) = 0$, then~$\Db(\cO,C)$ is equivalent to the derived category of the quiver
\begin{equation*}
\xymatrix@1{ \bullet \ar[r] & \bullet \ar@<.5ex>[r] \ar@<-.5ex>[r] & \bullet} 
\end{equation*}
with no relations.
In particular, $\Db(\cO,C)$ has a full exceptional collection.
\end{lemma}

For~$g > 0$ the category~$\Db(\cO,C)$ does not admit a full exceptional collection
(e.g., by Proposition~\ref{prop:invariants-ac});
in particular, it is not equivalent to the derived category of a directed quiver.

The next lemma gives an alternative description of the augmentation of an elliptic curve.

\begin{lemma}
\label{lem:ac-1}
If~$\g(C) = 1$, then~$\Db(\cO,C)$ is derived equivalent to the square root stack
\begin{equation*}
\Db(\cO,C) \simeq \Db(\sqrt{C,x}),
\end{equation*}
where~$x \in C$ is a $\kk$-point of~$C$.
\end{lemma}

\begin{proof}
For a curve~$C$ of genus~1 there is an autoequivalence
which takes~$\cO_C$ to the structure sheaf of a point~$\cO_x$.
By Theorem~\ref{thm:gluing} it induces an equivalence of~$\Db(\cO,C)$
onto the gluing of~$\Db(\kk)$ and~$\Db(C)$ with the gluing object~$\cO_x$,
which by Corollary~\ref{cor:gluing-auto} and Lemma~\ref{lem:point-augmentation} is equivalent to~$\Db(\sqrt{C,x})$.
\end{proof}

For curves of genus~$g \in \{2,3,4\}$ no direct geometric descriptions of augmentations are available.
However, they have some interesting properties that we want to point out.

\begin{lemma}
\label{lem:ac-3}
Let~$C$ be a hyperelliptic curve of genus~$g = 3$.
If~$\cL$ is the hyperelliptic line bundle, the augmentation~$\fa(\cL)$ is a $2$-spherical object, i.e.,
\begin{equation*}
\Ext^\bullet(\fa(\cL), \fa(\cL)) \cong \kk \oplus \kk[-2]
\qquad\text{and}\qquad 
\bS_{\Db(\cO,C)}(\fa(\cL)) \cong \fa(\cL)[2].
\end{equation*}
In particular, the functor~$\cF \mapsto \Cone\Big(\Ext^\bullet(\fa(\cL),\cF) \otimes \fa(\cL) \to \cF\Big)$
is an autoequivalence of~$\Db(\cO,C)$.
\end{lemma}

\begin{proof}
The first isomorphism follows from the argument of Proposition~\ref{prop:bnp-exceptional}
(because the Petri map for~$\cL$ is surjective),
and the second follows easily from Lemma~\ref{lem:serre-lower-bound}
applied to~$\cF_1 = \cF_2 = \cL$.
\end{proof}

For~$g = 4$ instead of a spherical object we obtain a spherical pair (see~\cite[\S2.2]{KP21}).

\begin{lemma}
\label{lem:ac-4}
Let~$C$ be a non-hyperelliptic curve of genus~$g = 4$ lying in a smooth quadric~\mbox{$\P^1 \times \P^1 \subset \P^3$}.
If~$\cL_1 = \cO_{\P^1 \times \P^1}(1,0)\vert_C$ and~$\cL_2 = \cO_{\P^1 \times \P^1}(0,1)\vert_C$ are the trigonal line bundles on~$C$,
then the augmentations~$\cE_{\cL_i} = \fa(\cL_i)$ are exceptional objects and form a spherical pair, i.e., 
\begin{equation*}
\Ext^\bullet(\cE_{\cL_1}, \cE_{\cL_2}) \cong 
\Ext^\bullet(\cE_{\cL_2}, \cE_{\cL_1}) \cong \kk[-2],
\quad
\bS_{\Db(\cO,C)}(\cE_{\cL_1}) \cong \cE_{\cL_2}[2]
\quad\text{and}\quad 
\bS_{\Db(\cO,C)}(\cE_{\cL_2}) \cong \cE_{\cL_1}[2].
\end{equation*}
In particular, the functor~$\cF \mapsto \Cone\Big(\bigoplus \Ext^\bullet(\cE_{\cL_i},\cF) \otimes \cE_{\cL_i} \to \cF\Big)$
is an autoequivalence of~$\Db(\cO,C)$.
\end{lemma}

\begin{proof}
Exceptionality of~$\cE_{\cL_i}$ is proved in Proposition~\ref{prop:bnp-exceptional}.
The same argument proves the first isomorphisms,
and the last two follow easily from Lemma~\ref{lem:serre-lower-bound}
applied to~$\cF_i = \cL_i$.
\end{proof}

Finally, we discuss the most interesting example.
Let~$C$ be a non-trigonal curve of genus~$g = 5$.
By the Enriques--Babbage theorem, the canonical model~$C \subset \P^4$ is a complete intersection of a net of quadrics.
Recall from~\cite{K08} the construction of the even Clifford algebra~$\Cliff_0$ on~$\P^2$ associated with such a net.

\begin{proposition}
\label{prop:ac-g5}
Assume the characteristic of~$\kk$ is not~$2$.
If~$C$ is a smooth non-trigonal curve of genus~$g = 5$, then
\begin{equation*}
\Db(\cO,C) \simeq \Db(\P^2, \Cliff_0),
\end{equation*}
where~$\Cliff_0$ is the sheaf of even parts of Clifford algebras corresponding to the quadrics in the net of~$C$
and~$\Db(\P^2,\Cliff_0)$ is the derived category of sheaves of~$\Cliff_0$-modules on~$\P^2$.

Moreover, if~$C$ has no degenerate even theta-characteristics over~$\bar\kk$,
then~$\Db(\cO,C)$ is equivalent to a twisted derived category of the root stack~$\sqrt{\P^2,\Gamma}$,
where~$\Gamma \subset \P^2$ is the discriminant curve of the net.
\end{proposition}

\begin{proof}
The equivalence of~$\Db(\P^2,\Cliff_0)$ with the augmentation of~$C$ is proved in Proposition~\ref{prop:gluing-quadrics}.

Now assume that~$C$ has no degenerate even theta-characteristics over~$\bar\kk$.
Then no quadric in the net has corank~2.
Indeed, if~$C \subset \P^4$ lies on a quadric of corank~$2$,
then the projection out of its vertex~$\P^1$ 
(the curve~$C$ does not intersect the vertex, because it is a smooth complete intersection)
defines a covering~\mbox{$C_{\bar\kk} \to \P^1$} of degree~4
and the pullback of~$\cO_{\P^1}(1)$ is a degenerate even theta-characteristic.
Therefore, the net of quadrics has only ``simple degenerations along~$\Gamma$'', 
where~$\Gamma \subset \P^2$ is the discriminant curve,
hence the results of~\cite[Section~3.6]{K08} allow us to identify the category~$\Db(\P^2, \Cliff_0)$ 
with a twisted category of the root stack~$\sqrt{\P^2,\Gamma}$
(where the twist is given by an appropriate sheaf of Azumaya algebras).
\end{proof}

\begin{remark}
Under the equivalence of Proposition~\ref{prop:ac-g5} 
the structure sheaf of a point in~\mbox{$\P^2 \setminus \Gamma \subset \sqrt{\P^2,\Gamma}$}
corresponds to the augmentation~$\fa(\cS\vert_C)$ 
of the restriction of the rank~2 spinor bundle~$\cS$ on the corresponding smooth quadric~$Q$ containing~$C$,
and the structure sheaves of ``half-points'' over a point in~$\Gamma$ 
correspond to the augmentations of the pullbacks of the line bundles~$\cO(1,0)$ and~$\cO(0,1)$
from the base of the corresponding quadratic cone containing~$C$.
\end{remark}

\begin{remark}
Every point of the curve~$C$ gives a regular isotropic section 
for the 3-dimensional quadric bundle~$\cQ \to \P^2$ obtained from the net of quadrics.
Applying the hyperbolic reduction, we obtain a conic bundle~$\cC \to \P^2$ 
whose even Clifford algebra is Morita equivalent to the algebra ~$\Cliff_0$ in Proposition~\ref{prop:ac-g5}.
The identification of Proposition~\ref{prop:ac-g5} shows that~$\Db(\cO,C)$ is a ``non-commutative del Pezzo surface''
in the sense of Chan and Ingalls, see~\cite[\S8]{CI}.
\end{remark}

\section{Reduced ideal point gluing of curves}
\label{sec:rip-gluing}

In this section we study a more complicated example of gluing ---
the gluing of two curves whose gluing bimodule is
isomorphic to the ideal of a point.

\begin{definition}
\label{def:point-gluing}
{\sf An ideal point gluing of curves}~$C_1$ and~$C_2$ 
with respect to $\kk$-points~$x_1 \in C_1$, $x_2 \in C_2$
is a triangulated category~$\Db(C_1,C_2;x_1,x_2)$ that admits a semiorthogonal decomposition
\begin{equation*}
\Db(C_1,C_2;x_1,x_2) = \langle \Db(C_1), \Db(C_2) \rangle
\end{equation*}
such that
\begin{equation}
\label{eq:ext-pg}
\RHom_{\Db(C_1,C_2;x_1,x_2)}(\cF_1,\cF_2) \cong 
\rH^\bullet(C_1 \times C_2, (\cF_1^\vee \boxtimes \cF_2) \otimes \cI_{(x_1,x_2)})
\end{equation}
for any~$\cF_1 \in \Db(C_1)$, $\cF_2 \in \Db(C_2)$, 
where~$\cI_{(x_1,x_2)}$ is the ideal sheaf of the point~$x \coloneqq (x_1,x_2) \in C_1 \times C_2$.
In other words, the category~$\Db(C_1,C_2;x_1,x_2)$ is the gluing of the categories~$\Db(C_1)$ and~$\Db(C_2)$ 
whose gluing object~$\urG \in \Db(C_1 \times C_2)$ is isomorphic to the ideal sheaf~$\cI_{(x_1,x_2)}$.
\end{definition}

Usually we abbreviate the notation~$\Db(C_1,C_2;x_1,x_2)$ to simply~$\Db(C_1,C_2)$.

\begin{remark}
\label{rem:pg-non-symmetric}
If we mutate the semiorthogonal decomposition~$\Db(C_1,C_2) = \langle \Db(C_1), \Db(C_2) \rangle$
and realize~$\Db(C_1,C_2)$ as a gluing of~$\Db(C_2)$ and~$\Db(C_1)$, 
the gluing object will be isomorphic to~$\cI_{(x_1,x_2)}^\vee$, 
which is a complex with two cohomology sheaves: $\cO_{C_1 \times C_2}$ in degree~0 and~$\cO_{x_1,x_2}$ in degree~1,
see Remark~\ref{rem:gluing-dual}.
Therefore, Definition~\ref{def:point-gluing} is not symmetric with respect to~$C_1$ and~$C_2$.
\end{remark}

\subsection{Basic properties and exotic exceptional object}

The basic invariants of an ideal point gluing can be computed 
with the aid of Propositions~\ref{prop:gluing-sp}, \ref{prop:invariants-general} and~\ref{prop:ij-gluing}.

\begin{proposition}
\label{prop:invariants-pg}
Let~$C_1,C_2$ be smooth proper curves of genus~$g_1, g_2$
and let~$\Db(C_1,C_2)$ be an ideal point gluing of~$\Db(C_i)$.
Then~$\Db(C_1,C_2)$ is a smooth and proper triangulated category.
Moreover, we have
\begin{arenumerate}
\item 
\label{it:additive-pg}
$\HH_\bullet(\Db(C_1,C_2)) = \kk^{g_1 + g_2}[1] \oplus \kk^4 \oplus \kk^{g_1 + g_2}[-1]$
and
$\rK_0(\Db(C_1,C_2)) = \Z^4 \oplus \Pic^0(C_1) \oplus \Pic^0(C_2)$.
\item 
\label{it:rkn-pg}
$\rKn(\Db(C_1,C_2)) = \Z^4$ and in the basis~$[\cO_{C_1}]$, $[\cO_{x_1}]$, $[\cO_{C_2}]$, $[\cO_{x_2}]$ 
\textup(where we use the embedding functors~$\Db(C_i) \hookrightarrow \Db(C_1,C_2)$
to realize~$\cO_{C_i}$ and~$\cO_{x_i}$ as objects of~$\Db(C_1,C_2)$\textup)
the matrix of the Euler form is
\begin{equation*}
\upchi_{\Db(C_1,C_2)} = 
\begin{pmatrix} 
1 - g_1 & 1 & g_1g_2 - g_1 - g_2 & 1 - g_1 \\ 
-1 & 0 & g_2 - 1 & -1 \\ 
0 & 0 & 1 - g_2 & 1 \\ 
0 & 0 & -1 & 0 
\end{pmatrix}.
\end{equation*}
\item 
\label{it:jac-pg}
If~$\kk = \CC$, the Euler pairing on~$\rK^\tp_1(\Db(C_1,C_2))$ is unimodular, skew-symmetric, and positive definite,
and~$\Jac(\Db(C_1,C_2)) \cong \Jac(C_1) \times \Jac(C_2)$ is an isomorphism of principally polarized abelian varieties.
\item 
\label{it:hhc-pg}
If~$g_1, g_2 \ge 2$, then~$\HH^\bullet(\Db(C_1,C_2)) = \kk \oplus \kk^{3(g_1 + g_2)-4}[-2] \oplus \kk^{g_1g_2}[-3]$.
\end{arenumerate}
\end{proposition}

\begin{proof}
The smoothness and properness of~$\Db(C_1,C_2)$ are proved in Proposition~\ref{prop:gluing-sp}.

Parts~\ref{it:additive-pg} and~\ref{it:rkn-pg} of the proposition
follow immediately from Proposition~\ref{prop:invariants-general}.

To prove part~\ref{it:jac-pg}, we apply Proposition~\ref{prop:gluing-geometric} with~$\bcXo = \P^3$.
It shows that~$\Db(C_1,C_2)$ can be realized as the complement to an exceptional collection in the derived category
of the consecutive blowup~$\hcXo$ of~$\P^3$ with centers isomorphic to~$C_1$ and~$C_2$.
Therefore, \cite[Proposition~5.23]{Per22} gives an isometry
\begin{equation*}
\rK^\tp_1(\Db(\hcXo)) \cong \rK^\tp_1(\Db(C_1,C_2))
\end{equation*}
and proves that the Euler pairing on~$\rK^\tp_1(\Db(C_1,C_2))$ is unimodular, skew-symmetric, and positive definite.
Since the same is true for the Euler pairings on~$\rK^\tp_1(\Db(C_i))$, Proposition~\ref{prop:ij-gluing}
implies the required isomorphism of the intermediate Jacobians.

To prove part~\ref{it:hhc-pg}, we note that~$\HH^\bullet(\Db(C_i)) = \kk \oplus \kk^{g_i}[-1] \oplus \kk^{3g_i - 3}[-2]$,
while
\begin{equation*}
\Ext^\bullet(\urG, \urG) \cong 
\Ext^\bullet_{C_1 \times C_2}(\cI_{(x_1,x_2)}, \cI_{(x_1,x_2)}) \cong 
\kk \oplus \kk^{g_1 + g_2 + 2}[-1] \oplus \kk^{g_1g_2}[-2].
\end{equation*}
By Proposition~\ref{prop:invariants-general}\ref{it:hh-triangle}
it remains to note that the map
\begin{multline*}
\rH^1(C_1, \cO_{C_1}) \oplus \rH^1(C_2, \cO_{C_2}) = 
\HH^1(\Db(C_1)) \oplus \HH^1(\Db(C_2)) \\ \to 
\Ext^1(\cI_{(x_1,x_2)}, \cI_{(x_1,x_2)}) = 
\rH^1(C_1, \cO_{C_1}) \oplus \rH^1(C_2, \cO_{C_2}) \oplus \Ext^1(\cO_{(x_1,x_2)}, \cO_{(x_1,x_2)})
\end{multline*}
is injective (e.g., by the argument of Proposition~\ref{prop:invariants-ac}\ref{it:hhc-ac}), and to show that the map
\begin{multline}
\label{eq:hh2-ext2}
\rH^1(C_1, \cT_{C_1}) \oplus \rH^1(C_2, \cT_{C_2}) = 
\HH^2(\Db(C_1)) \oplus \HH^2(\Db(C_2)) \\ \to 
\Ext^2(\cI_{(x_1,x_2)}, \cI_{(x_1,x_2)}) = 
\rH^1(C_1, \cO_{C_1}) \otimes \rH^1(C_2, \cO_{C_2})
\end{multline}
is zero.
Indeed, we have an isomorphism~$\rH^1(C_1, \cT_{C_1}) \oplus \rH^1(C_2, \cT_{C_2}) \cong \rH^1(C_1 \times C_2, \cT_{C_1 \times C_2})$,
and it is easy to see that the map~\eqref{eq:hh2-ext2} is induced by multiplication with the Atiyah class
\begin{equation*}
\At_{\cI_{(x_1,x_2)}} \in \Ext^1(\cT_{C_1 \times C_2}, \cI_{(x_1,x_2)}^\vee \otimes \cI_{(x_1,x_2)}).
\end{equation*}
Furthermore, since~$\Ext^2(\cI_{(x_1,x_2)}, \cI_{(x_1,x_2)}) \cong \Ext^2(\cO_{C_1 \times C_2}, \cO_{C_1 \times C_2})$,
using functoriality of the Atiyah class and the vanishing of~$\At_{\cO_{C_1 \times C_2}}$,
one can show that the map~\eqref{eq:hh2-ext2} is zero.
\end{proof}

\begin{remark}
The computation in~Proposition~\ref{prop:invariants-pg}\ref{it:hhc-pg} shows that
\begin{equation*}
\HH^2(\Db(C_1,C_2)) = 
\rH^1(C_1, \cT_{C_1}) \oplus \rH^1(C_2, \cT_{C_2}) \oplus \cT_{C_1,x_1} \oplus \cT_{C_2,x_2}.
\end{equation*}
This means that any (infinitesimal) deformation of the category~$\Db(C_1,C_2)$ 
is obtained either by moving the points~$x_i$ on the curves~$C_i$
or by deforming the curves~$C_i$.
Moreover, since~$\HH^1(\Db(C_1,C_2)) = 0$, the group of autoequivalences of~$\Db(C_1,C_2)$ is discrete.
\end{remark}

The most surprising property of an ideal point gluing of curves is that it carries an exotic exceptional object.
Recall Notation~\ref{not:fff}.

\begin{theorem}
\label{thm:exotic-exceptional-curves}
Let~$\Db(C_1,C_2)$ be the ideal point gluing of the curves~$C_1$, $C_2$ 
with respect to points~$x_1 \in C_1$ and~$x_2 \in C_2$.
There is a canonical element~$\eps \in \rH^0(\rG(\cO_{x_1}, \cO_{x_2})) = \Hom_{\Db(C_1,C_2)}(\cO_{x_1}, \cO_{x_2})$ such that 
\begin{equation}
\label{eq:re}
\rE \coloneqq (\cO_{x_1},\cO_{x_2},\eps) \cong \Cone\Big( \cO_{x_1} \xrightarrow{\ \eps\ } \cO_{x_2} \Big)[-1]
\end{equation}
is an exceptional object in~$\Db(C_1,C_2)$.
\end{theorem}

\begin{proof}
Recall that the ideal~$\cI_{(x_1,x_2)}$ has a simple resolution
\begin{equation*}
0 \to \cI_{(x_1,x_2)} \to \cO_{C_1 \times C_2} \to \cO_{x_1,x_2} \to 0.
\end{equation*}
Both terms are decomposable sheaves: $\cO_{C_1 \times C_2} \cong \cO_{C_1} \boxtimes \cO_{C_2}$,
$\cO_{x_1,x_2} \cong \cO_{x_1} \boxtimes \cO_{x_2}$.
Using the K\"unneth formula and the isomorphism~$\cO_{x_2}^\vee \cong \cO_{x_2}[-1]$, 
we can rewrite the gluing bimodule of~$\Db(C_1,C_2)$ as
\begin{equation}
\label{eq:pg}
\rG(\cF_1,\cF_2) \cong 
\Cone\Big( 
\RHom(\cF_1,\cO_{C_1}) \otimes \RHom(\cO_{C_2},\cF_2)[-1]
\xrightarrow{\ \ \ }
\RHom(\cF_1,\cO_{x_1}) \otimes \RHom(\cO_{x_2},\cF_2) 
\Big).
\end{equation}
This suggests that for the argument only the interactions between the sheaves~$\cO_{C_i}$ and~$\cO_{x_i}$ are important.
This is indeed the case, so we start by recalling the $\Ext$-spaces between these sheaves and introducing some notation.
Using Serre duality on~$C_i$, we find
\begin{itemize}
\item 
$\RHom(\cO_{C_i}, \cO_{x_i}) \cong \kk$,
and we denote by~$\lambda_i \colon \cO_{C_i} \to \cO_{x_i}$ a generator;
\item 
$\RHom(\cO_{x_i}, \cO_{C_i}) \cong \kk[-1]$,
and we denote by~$\mu_i \colon \cO_{x_i} \to \cO_{C_i}[1]$ a generator;
\item 
$\RHom(\cO_{x_i}, \cO_{x_i}) \cong \kk \oplus \kk[-1]$.
\end{itemize}
Note that the space~$\Ext^1(\cO_{x_i}, \cO_{x_i})$ is generated by the composition~$\lambda_i \circ \mu_i$
and the morphism in~\eqref{eq:pg} is given by~$\lambda_1 \otimes \mu_2$.
We consider~$\lambda_i$ and~$\mu_i$ as elements of degree~0 and~1, respectively,
so the following are isomorphisms of graded vector spaces:
\begin{equation*}
\RHom(\cO_{C_i}, \cO_{x_i}) \cong \kk \cdot \lambda_i,
\quad
\RHom(\cO_{x_i}, \cO_{C_i}) \cong \kk \cdot \mu_i,
\quad\text{and}\quad
\RHom(\cO_{x_i}, \cO_{x_i}) \cong \kk \cdot \id_{\cO_{x_i}} \oplus \kk \cdot (\lambda_i \circ \mu_i).
\end{equation*}

Now, using the above notation, we check the following

\begin{claim*}
There is an isomorphism
\begin{equation}
\label{eq:prg-rk1-rk2}
\rG(\cO_{x_1}, \cO_{x_2}) \cong
\kk \oplus \kk[-1] \oplus \kk[-1]
\end{equation}
and under the natural left and right actions of~$\RHom(\cO_{x_i},\cO_{x_i})$ on~$\rG(\cO_{x_1}, \cO_{x_2})$
the generators~$\lambda_i \circ \mu_i$ of~$\Ext^{1}(\cO_{x_i}, \cO_{x_i})$
take a generator~$\eps \in \rG(\cO_{x_1}, \cO_{x_2})$ of the first summand in the right-hand side of~\eqref{eq:prg-rk1-rk2}
to generators of the second and third summands, respectively.
\end{claim*}

Indeed, the source of the morphism in the right-hand side of~\eqref{eq:pg} with~$\cF_1 = \cO_{x_1}$ and~$\cF_2 = \cO_{x_2}$
that computes~$\rG(\cO_{x_1}, \cO_{x_2})$ is the shifted tensor product
of~$\RHom(\cO_{x_1}, \cO_{C_1})$ and~$\RHom(\cO_{C_2}, \cO_{x_2})$,
and the target is the tensor product of~$\RHom(\cO_{x_1}, \cO_{x_1})$ and~$\RHom(\cO_{x_2}, \cO_{x_2})$,
while the map between the tensor products is given by~$\lambda_1 \otimes \mu_2$.
We conclude that~$\rG(\cO_{x_1},\cO_{x_2})$ is the cone of the map
\begin{equation*}
\kk \cdot (\mu_1 \otimes \lambda_2)[-1]
\xrightarrow{\ \lambda_1 \otimes \mu_2\ }
\kk \cdot (\id_{\cO_{x_1}} \otimes \id_{\cO_{x_2}}) \oplus
\kk \cdot (\id_{\cO_{x_1}} \otimes (\lambda_2 \circ \mu_2)) \oplus
\kk \cdot ((\lambda_1 \circ \mu_1) \otimes \id_{\cO_{x_2}}) \oplus
\kk \cdot ((\lambda_1 \circ \mu_1) \otimes (\lambda_2 \circ \mu_2)). 
\end{equation*}
The map is an isomorphism of the source onto the last summand of the target, 
hence its cone is isomorphic to the sum of the first three summands of the target.
The statement about the action of~$\lambda_i \circ \mu_i$ also follows.

\medskip

Now, finally, we can check that the object~$\rE \in \Db(C_1,C_2)$ defined from the distinguished triangle
\begin{equation}
\label{eq:ee-triangle}
\rE \xrightarrow{\quad} \cO_{x_1} \xrightarrow{\ \eps\ } \cO_{x_2},
\end{equation}
where~$\eps$ is a generator of the first summand in~\eqref{eq:prg-rk1-rk2}, is exceptional.
For this we consider the triangle of Lemma~\ref{lem:hom-gluing},
which takes the form
\begin{equation*}
\dots \xrightarrow\quad
\Ext^\bullet_{\Db(C_1,C_2)}(\rE,\rE) \xrightarrow\quad
\Ext^\bullet_{\Db(C_1)}(\cO_{x_1},\cO_{x_1}) \oplus \Ext^\bullet_{\Db(C_2)}(\cO_{x_2},\cO_{x_2}) \xrightarrow{\ \eps\ }
\rH^\bullet(\rG(\cO_{x_1}, \cO_{x_2})) \xrightarrow\quad \dots.
\end{equation*}
The middle term is equal to the sum of
\begin{equation*}
\kk \cdot \id_{\cO_{x_1}} \ \oplus\  \kk \cdot (\lambda_1 \circ \mu_1)
\qquad\text{and}\qquad 
\kk \cdot \id_{\cO_{x_2}} \ \oplus\  \kk \cdot (\lambda_2 \circ \mu_2);
\end{equation*}
the right term is given by~\eqref{eq:prg-rk1-rk2};
and, as we checked in the above claim, the map~$\eps$ acts as follows:
\begin{equation*}
\id_{\cO_{x_1}} \mapsto \eps,
\qquad 
\id_{\cO_{x_2}} \mapsto -\eps,
\qquad 
\lambda_1 \circ \mu_1 \mapsto (\lambda_1 \circ \mu_1) \otimes \id_{\cO_{x_2}},
\qquad 
\lambda_2 \circ \mu_2 \mapsto -\id_{\cO_{x_1}} \otimes (\lambda_2 \circ \mu_2).
\end{equation*}
It follows that~$\Ext_{\Db(C_1,C_2)}^\bullet(\rE,\rE) \cong \kk$ and~$\rE$ is exceptional.
\end{proof}

\begin{remark}
\label{rem:generalization}
As we noticed in the proof, the computation relies only on the structure of~$\RHom$-spaces between~$\cO_{C_i}$ and~$\cO_{x_i}$.
More precisely, it uses the \emph{spherical property} of~$\cO_{x_i}$, 
and the \emph{adherence property} between~$\cO_{C_i}$ and~$\cO_{x_i}$.
Consequently, the same idea can be used to construct an exotic exceptional object 
in a gluing~$\cD$ of two triangulated categories~$\cD_1$ and~$\cD_2$ if the gluing bimodule can be written as
\begin{equation*}
\rG(\cF_1,\cF_2) \cong 
\Cone\Big(
\RHom_{\cD_1}(\cF_1,\rL_1) \otimes \RHom_{\cD_2}(\rL_2,\cF_2) \to
\RHom_{\cD_1}(\cF_1,\rK_1) \otimes \RHom_{\cD_2}(\rK_2,\cF_2) 
\Big),
\end{equation*}
where
\begin{itemize}
\item 
$\rK_i$ is a spherical object in~$\cD_i$, 
\item 
$\rL_1$ is {\sf left adherent} to~$\rK_1$, i.e., $\RHom(\rL_1,\rK_1) = \kk$, and
\item 
$\rL_2$ is {\sf right adherent} to~$\rK_2$, i.e., $\RHom(\rK_2,\rL_2) = \kk$.
\end{itemize}
In this case there is still a canonical element~$\eps \in \rH^0(\rG(\rK_1,\rK_2))$
(coming from the identity morphisms of~$\rK_1$ and~$\rK_2$),
and the object~$(\rK_1, \rK_2, \eps) \in \cD$ is exceptional.
\end{remark}

\subsection{Reduced ideal point gluing of curves}
\label{ss:rpg}

In the rest of this section we discuss the subcategory of~$\Db(C_1,C_2)$ defined as follows:

\begin{definition}
\label{def:rpg}
Let~$\Db(C_1,C_2)$ be an ideal point gluing of two smooth proper curves as in Definition~\ref{def:point-gluing}
and let~$\rE \in \Db(C_1,C_2)$ be the exotic exceptional object constructed in Theorem~\ref{thm:exotic-exceptional-curves}.
The {\sf reduced ideal point gluing} of~$C_1$ and~$C_2$ is defined as the triangulated category
\begin{equation}
\label{eq:rpg-def}
\RPG{C_1,C_2}
\coloneqq {}^\perp\rE \subset \Db(C_1,C_2),
\end{equation}
i.e., the orthogonal complement in~$\Db(C_1,C_2)$ to the exceptional object~$\rE$.
\end{definition}

By definition, the category~$\Db(C_1,C_2)$ has two (incompatible if~$g_i \ge 1$) semiorthogonal decompositions
\begin{equation}
\label{eq:rpg-sod}
\langle \Db(C_1), \Db(C_2) \rangle = \langle \rE, \RPG{C_1,C_2} \rangle.
\end{equation} 
In particular, \eqref{eq:rpg-sod} provides a new counterexample to the Jordan--H\"older property for semiorthogonal decompositions
(for other counterexamples, see~\cite{K13, BBS, Krah}).

In the next proposition we compute the basic invariants of the reduced ideal point gluing of curves.

\begin{proposition}
\label{prop:invariants-rpg}
Let~$C_1,C_2$ be smooth curves of genus~$g_1, g_2$ and let~$\RPG{C_1,C_2}$
be a reduced ideal point gluing of~$C_1$ and~$C_2$.
Then~$\RPG{C_1,C_2}$ is a smooth and proper triangulated category.
Moreover, we have
\begin{arenumerate}
\item 
\label{it:additive-rpg}
$\HH_\bullet(\RPG{C_1,C_2}) = \kk^{g_1 + g_2}[1] \oplus \kk^3 \oplus \kk^{g_1 + g_2}[-1]$
and
$\rK_0(\RPG{C_1,C_2}) = \Z^3 \oplus \Pic^0(C_1) \oplus \Pic^0(C_2)$.
\item 
\label{it:rkn-rpg}
$\rKn(\RPG{C_1,C_2}) = \Z^3$ and in the basis~$[\cO_{C_2}] + [\cO_{x_1}] - [\cO_{x_2}]$, 
$[\cO_{C_1}] - g_1[\cO_{x_1}] - g_2[\cO_{x_2}]$, $[\cO_{C_2}] + [\cO_{x_1}]$
the matrix of the Euler form is
\begin{equation*}
\upchi_{\RPG{C_1,C_2}} = 
\begin{pmatrix}
1 & -1 & 1 \\
0 & 1 - g_1 - g_2 & 1 \\
0 & -1 & 0
\end{pmatrix}.
\end{equation*}
\item 
\label{it:jac-rpg}
If~$\kk = \CC$, the Euler pairing on~$\rK^\tp_1(\RPG{C_1,C_2})$ is unimodular, skew-symmetric, and positive definite,
and~$\Jac(\RPG{C_1,C_2}) \cong \Jac(C_1) \times \Jac(C_2)$ is an isomorphism of principally polarized abelian varieties.
\item 
\label{it:hhc-rpg}
If~$g_1, g_2 \ge 2$, then~$\HH^\bullet(\RPG{C_1,C_2}) = \kk \oplus \kk^{3(g_1 + g_2)-3}[-2] \oplus \kk^{g_1g_2}[-3] \oplus \kk^{g_1g_2}[-4]$.
\end{arenumerate}
\end{proposition}

\begin{proof}
The smoothness and properness of~$\RPG{C_1,C_2}$
follow from Propositions~\ref{prop:gluing-sp} and~\ref{prop:invariants-pg}.

Part~\ref{it:additive-rpg} follows from additivity of Hochschild homology and~$\rK_0$
and Proposition~\ref{prop:invariants-pg}\ref{it:additive-pg}.

Part~\ref{it:rkn-rpg} follows from Proposition~\ref{prop:invariants-pg}\ref{it:rkn-pg}.

Part~\ref{it:jac-rpg} follows from Proposition~\ref{prop:invariants-pg}\ref{it:jac-pg} and Remark~\ref{rem:ppav}.

Finally, to compute the Hochschild cohomology of~$\RPG{C_1,C_2} {} = {}^\perp\rE$, we use the distinguished triangle
\begin{equation}
\label{eq:hhc-rpg-triangle}
\Ext^\bullet(\bS_{\Db(C_1,C_2)}(\rE), \rE) \to
\HH^\bullet(\Db(C_1,C_2)) \to 
\HH^\bullet(\RPG{C_1,C_2})
\end{equation}
constructed in~\cite[Theorem~3.3 and Proposition~3.7]{K15}, similarly to the case of Remark~\ref{rem:hhc-bnp-mod}.
Theorem~\ref{thm:serre-general} implies that~$\bS_{\Db(C_1,C_2)}(\rE) = (\bar\cF_1,\bar\cF_2,\bar\phi)$,
where
\begin{equation*}
\bar\cF_1 = \rH^0(C_2,\omega_{C_2}(x_2)) \otimes \cO_{C_1}[3]
\qquad\text{and}\qquad
\bar\cF_2 = \omega_{C_2}(x_2)[3],
\end{equation*}
and then a direct computation shows that
\begin{equation}
\label{eq:ext-se-e}
\Ext^\bullet(\bS_{\Db(C_1,C_2)}(\rE),\rE) \cong 
\kk[-3] \oplus \rH^0(C_2,\omega_{C_2})^\vee \otimes \rH^0(C_1,\omega_{C_1})^\vee[-5] \cong
\kk[-3] \oplus \kk^{g_1g_2}[-5].
\end{equation}
It follows that there is an exact sequence
\begin{equation*}
0 \to \HH^2(\Db(C_1,C_2)) \to \HH^2(\RPG{C_1,C_2}) \to \kk \to \HH^3(\Db(C_1,C_2)) \to \HH^3(\RPG{C_1,C_2}) \to 0.
\end{equation*}
As we will show in Theorem~\ref{thm:main}, the category~$\RPG{C_1,C_2}$ 
is a smooth and proper limit of augmentations of curves of genus~$g_1 + g_2$,
hence by semicontinuity of Hochschild cohomology and Proposition~\ref{prop:invariants-ac}\ref{it:hhc-ac}, 
we must have~$\dim(\HH^2(\RPG{C_1,C_2})) \ge 3(g_1 + g_2) - 3 > 3(g_1 + g_2) - 4 = \dim(\HH^2(\Db(C_1,C_2)))$,
hence the morphism~$\kk \to \HH^3(\Db(C_1,C_2))$ must be zero, and part~\ref{it:hhc-rpg} follows.
\end{proof}

\begin{remark}
In the proof of Proposition~\ref{prop:invariants-rpg}\ref{it:hhc-rpg} 
we used the existence of a deformation of~$\RPG{C_1,C_2}$ to the augmentation~$\Db(\cO,C)$ of a curve~$C$ of genus~$g_1 + g_2$.
On the other hand, the equality
\begin{equation*}
\dim(\HH^2(\RPG{C_1,C_2})) = 3(g_1 + g_2) - 3 = \dim(\HH^2(\Db(\cO,C)))
\end{equation*}
suggests that the reduced ideal point gluings form a boundary component of a ``moduli space'' of augmented curves
contained in the smooth locus of the ``moduli space''.
This is a categorical incarnation of the boundary component~$\cM_{g_1,1} \times \cM_{g_2,1} \subset \overline{\cM}_{g_1 + g_2}$
of the Deligne--Mumford compactification of~$\cM_{g_1 + g_2}$.
\end{remark}

\begin{remark}
\label{rem:rpg-nonzero}
Proposition~\ref{prop:invariants-rpg}\ref{it:rkn-rpg} 
implies that the group~$\rKn(\RPG{C_1,C_2})$ is isometric to~$\rKn(\Db(\cO,C))$ 
(as abelian groups endowed with non-symmetric bilinear forms) if~$\g(C) = \g(C_1) + \g(C_2)$.
However, if~$\g(C_1),\g(C_2) \ge 2$, 
it follows from Propositions~\ref{prop:invariants-rpg}\ref{it:hhc-rpg} and~\ref{prop:invariants-ac}\ref{it:hhc-ac} 
that~$\RPG{C_1,C_2}$ is not equivalent to an augmented curve.
Moreover, the same is true if~$\g(C_1),\g(C_2) \ge 1$, because~\eqref{eq:ext-se-e} still holds,
and therefore~$\HH^4(\RPG{C_1,C_2}) \ne 0$.
\end{remark}

\begin{remark}
\label{rem:rpg-zero}
If~$\g(C_1) = 0$, then~$\cO_{C_1} \in \Db(C_1) \subset \Db(C_1,C_2)$ is exceptional.
On the other hand, it is easy to check that~$\rG^{12}(\cO_{x_2}) \cong \cO_{C_1}(-x_1) \oplus \cO_{x_1}$
and the morphism~$\eps$ in~\eqref{eq:re} corresponds to the embedding of the second summand,
which implies that~$\cO_{C_1} \in {}^\perp\rE = \RPG{C_1,C_2}$.
Moreover, we have
\begin{equation*}
\RHom_{\Db(C_1,C_2)}(\cO_{C_1}(1), \cF_2) \cong
\rH^\bullet(C_1 \times C_2, (\cO_{C_1}(-1) \boxtimes \cF_2) \otimes \cI_{(x_1,x_2)}) \cong
\cF_2\vert_{x_2}[-1] \cong
\RHom(\cF_2,\cO_{x_2})^\vee[-1]
\end{equation*}
and one can check that the functor
\begin{equation*}
\Db(C_2) \to {}^\perp\langle \rE, \cO_{C_1} \rangle,
\qquad 
\cF_2 \mapsto (\RHom(\cF_2,\cO_{x_2})^\vee[-1] \otimes \cO_{C_1}(1), \cF_2, \phi_{\cF_2}),
\end{equation*}
where~$\phi_{\cF_2}$ is the evaluation morphism~$\RHom_{\Db(C_1,C_2)}(\cO_{C_1}(1), \cF_2) \otimes \cO_{C_1}(1) \to \cF_2$,
is an equivalence of categories and conclude that~$\RPG{C_1,C_2} \simeq \Db(\cO,C_2)$ if~$\g(C_1) = 0$.
A similar argument proves that~$\RPG{C_1,C_2} \simeq \Db(\cO,C_1)$ if~$\g(C_2) = 0$.
\end{remark}

We expect that~$\RPG{C_1,C_2}$ does not contain any exceptional object if~$\g(C_1),\g(C_2) \ge 1$.

\begin{remark}
As we mentioned in Remark~\ref{rem:pg-non-symmetric}, the category~$\Db(C_1,C_2)$ is not symmetric with respect to~$C_1$ and~$C_2$.
However, it is possible that~$\RPG{C_1,C_2}$ is symmetric; it would be interesting to check whether this is the case.
\end{remark}

\section{Deformation}
\label{sec:smoothing}

In this section we construct a deformation relating a reduced ideal point gluing of curves~$C_1$ and~$C_2$
and augmentations of curves of genus~$\g(C_1) + \g(C_2)$.
For this we embed a 1-nodal curve~$C_1 \cup C_2$ into a rationally connected threefold,
study the derived category of its blowup, and then consider the smoothing of the threefold induced by a smoothing of the curve.

\subsection{The central fiber}
\label{ss:central-fiber}

Assume~$\bcXo$ is a smooth threefold such that its structure sheaf~$\cO_{\bcXo}$ is exceptional
(for instance, $\bcXo$ could be any rationally connected threefold).
Assume further that
\begin{equation*}
C = C_1 \cup C_2 \subset \bcXo
\end{equation*}
is a reducible 1-nodal curve.
We denote by~$x_1 \in C_1$ and~$x_2 \in C_2$ the points identified with the node of~$C$.
Consider the blowup
\begin{equation*}
\cXo \coloneqq \Bl_{C_1 \cup C_2}(\bcXo);
\end{equation*}
this is a 1-nodal threefold.
In the next lemma we construct a small resolution of~$\cXo$ 
and check that~$\cXo$ is maximally nonfactorial (\cite[Definition~6.10]{KS24}), 
i.e.,  that the restriction morphism from the Picard group of the blowup of~$\cXo$ at the node 
to the Picard group of its exceptional divisor is surjective.

\begin{lemma}
\label{lem:x-mnf}
Let~$\Bl_{C_1}(\bcXo)$ be the blowup of~$\bcXo$ at the component~$C_1$ of~$C$,
let~$C'_2 \subset \Bl_{C_1}(\bcXo)$ be the strict transform of the component~$C_2$,
and let~$\hcXo \coloneqq \Bl_{C'_2}(\Bl_{C_1}(\bcXo))$ be the blowup of~$\Bl_{C_1}(\bcXo)$ at~$C'_2$.
Then there is a commutative diagram
\begin{equation}
\label{eq:diagram-x}
\vcenter{\xymatrix{
\hcXo \coloneqq \Bl_{C'_2}(\Bl_{C_1}(\bcXo)) \ar[rr]^\varpi \ar[dr]_\pi && 
\Bl_{C_1 \cup C_2}(\bcXo) = \cXo \ar[dl]^{\rhoo}
\\
& \bcXo,
}}
\end{equation}
where~$\varpi$ is a small resolution of singularities.
The exceptional locus of~$\varpi$ is the strict transform~$L \subset \hcXo$ 
of the fiber of~$\Bl_{C_1}(\bcXo) \to \bcXo$ over the point~$x_0 \coloneqq C_1 \cap C_2$ in~$\bcXo$, and we have
\begin{equation}
\label{eq:ek-ell}
E_1 \cdot L = - 1
\qquad\text{and}\qquad 
E_2 \cdot L = 1,
\end{equation}
where~$E_1,E_2 \subset \hcXo$ are the exceptional divisors of~$\hcXo$ over~$C_1$ and~$C_2$, respectively.

In particular, the threefold~$\cXo = \Bl_{C_1 \cup C_2}(\bcXo)$ is maximally nonfactorial.
\end{lemma}

\begin{proof}
Consider the composition of the blowups
\begin{equation*}
\pi \colon \hcXo \coloneqq \Bl_{C'_2}(\Bl_{C_1}(\bcXo)) \to \Bl_{C_1}(\bcXo) \to \bcXo.
\end{equation*}
The scheme preimage of the curve~$C_1 \cup C_2 \subset \bcXo$ under~$\pi$ 
is the union~$E_1 \cup E_2$ of the exceptional divisors of~$\pi$, 
hence  the morphism~$\pi$ factors through the blowup~$\rhoo$,
giving the required commutative diagram.

The morphism~$\varpi$ is obviously an isomorphism over the complement of the point~$x_0 \in \bcXo$.
On the other hand, the fiber of~$\rhoo$ over~$x_0$ is a $\P^1$ (because~$C_1 \cup C_2$ is a local complete intersection),
while the fiber of~$\pi$ is the union of two smooth rational curves
\begin{equation*}
\pi^{-1}(x_0) = R \cup L,
\end{equation*}
where~$L$ is the strict transform of the fiber of the exceptional divisor of the first blowup~$\Bl_{C_1}(\bcXo) \to \bcXo$,
while~$R$ is the fiber of the second blowup~$\Bl_{C'_2}(\Bl_{C_1}(\bcXo)) \to \Bl_{C_1}(\bcXo)$ 
over the intersection point of the curve~$C'_2$ with~$E_1$.
It follows that~$\varpi$ is a small contraction; in particular, it is crepant.

Finally, we note that~$R$ is a fiber of~$E_2 \to C'_2$, while~$R \cup L$ is a fiber of~$E_1 \to C_1$,
hence
\begin{equation*}
E_2 \cdot R = -1,
\qquad 
E_1 \cdot R = 0,
\qquad
E_2 \cdot (R + L) = 0,
\qquad
E_1 \cdot (R + L) = -1.
\end{equation*}
In particular, \eqref{eq:ek-ell} follows.
Moreover, it follows that~$(E_1 + E_2) \cdot L = 0$, and since~$\varpi^*(\Pic(\cXo/\bcXo))$ is generated by~$E_1 + E_2$,
we conclude that~$L$ is contracted by~$\varpi$, hence coincides with its exceptional locus.

The maximal nonfactoriality of~$\cXo$ follows from~\eqref{eq:ek-ell} and~\cite[Lemma~6.14]{KS24}.
\end{proof}

In what follows we use freely notation introduced in Lemma~\ref{lem:x-mnf},
in particular the smooth rational curves~\mbox{$L,R \subset \hcXo$}.
Recall that~$L$ is the exceptional curve of the crepant contraction~$\varpi$;
in particular,
\begin{equation}
\label{eq:kl}
K_{\hcXo} \cdot L = 0,
\end{equation}
and since the singular point of~$\cXo$ is a node, we have~$\cN_{L/\hcXo} \cong \cO_L(-1) \oplus \cO_L(-1)$.
It follows that~$\cO_L(-1)$ is a 3-spherical object in~$\Db(\hcXo)$.
We denote by
\begin{equation}
\label{eq:bt-ol}
\bT_{\cO_L(-1)} \colon \Db(\hcXo) \to \Db(\hcXo),
\qquad 
\cF \mapsto \Cone \Big(\Ext^\bullet(\cO_L(-1),\cF) \otimes \cO_L(-1) \to \cF\Big)
\end{equation} 
the corresponding spherical twist, which is an autoequivalence of~$\Db(\hcXo)$.

Note that the curve~$R$ is the intersection of the exceptional divisors of~$\pi$, 
so that we have a cartesian (and, moreover, $\Tor$-independent) diagram
\begin{equation}
\label{eq:reex}
\vcenter{\xymatrix{
R \ar[r]^-{r_2} \ar[d]_{r_1} &
E_2 \ar[d]^{\varepsilon_2}
\\
E_1 \ar[r]^-{\varepsilon_1} &
\hcXo,
}}
\end{equation}
where~$\varepsilon_k$ and~$r_k$ are the natural embeddings.

Now we start describing the structure of the derived category of~$\hcXo$.
Since the structure sheaf~$\cO_{\bcXo}$ is exceptional by assumption,
we have a semiorthogonal decomposition
\begin{equation}
\label{eq:dbhx}
\Db(\bcXo) = \langle \cA_{\bcXo}, \cO_{\bcXo} \rangle,
\qquad
\text{where~$\cA_{\bcXo} \coloneqq \cO_{\bcXo}^\perp$.}
\end{equation}
Since~$\pi \colon \hcXo = \Bl_{C'_2}(\Bl_{C_1}(\bcXo)) \to \bcXo$ is a composition of two smooth blowups, we have
\begin{equation}
\label{eq:dbxp-1}
\Db(\hcXo) = \langle \pi^*(\cA_{\bcXo}), \cO_{\hcXo}, \Phi_1(\Db(C_1)), \Phi_2(\Db(C_2)) \rangle,
\end{equation}
where~$\Phi_k$ are the fully faithful embeddings of~$\Db(C_k)$ defined by
\begin{equation*}
\Phi_k \colon \Db(C_k) \to \Db(\hcXo),
\qquad 
\cF_k \mapsto \varepsilon_{k*}(p_k^*(\cF_k)),
\end{equation*}
where~$\varepsilon_k \colon E_k \hookrightarrow \hcXo$ and~$p_k \colon E_k \to C_k$ are the natural embedding and projection.

\begin{lemma}
For any~$\cF_1 \in \Db(C_1)$, $\cF_2 \in \Db(C_2)$, we have
\begin{equation}
\label{eq:hom-phi12}
\RHom_{\Db(\hcXo)}(\Phi_1(\cF_1), \Phi_2(\cF_2)) \cong 
\RHom_{\Db(\kk)}(\cF_1\vert_{x_1}, \cF_2\vert_{x_2})[-1],
\end{equation}
where~$\cF_k\vert_{x_k}$ stands for the derived restriction of~$\cF_k$ to the point~$x_k \in C_k$.
Furthermore,
\begin{equation}
\label{eq:hom-pi-phi12}
\RHom_{\Db(\hcXo)}(\cO_{\hcXo}, \Phi_k(\cF_k)) \cong
\RHom_{\Db(C_k)}(\cO_{C_k}, \cF_k).
\end{equation}
\end{lemma}

\begin{proof}
Using the $\Tor$-independent square~\eqref{eq:reex}, it is easy to check that
\begin{align*}
\RHom_{\Db(\hcXo)}(\Phi_1(\cF_1), \Phi_2(\cF_2)) &=
\RHom_{\Db(\hcXo)}(\varepsilon_{1*}(p_1^*(\cF_1)), \varepsilon_{2*}(p_2^*(\cF_2))) \\ & \cong
\RHom_{\Db(E_2)}(\varepsilon_2^*\varepsilon_{1*}(p_1^*(\cF_1)), p_2^*(\cF_2)) \\ & \cong
\RHom_{\Db(E_2)}(r_{2*}r_1^*(p_1^*(\cF_1)), p_2^*(\cF_2)) \\ & \cong
\RHom_{\Db(R)}(r_1^*(p_1^*(\cF_1)), r_2^!(p_2^*(\cF_2))) \\ & \cong
\RHom_{\Db(R)}(r_1^*(p_1^*(\cF_1)), r_2^*(p_2^*(\cF_2)) \otimes \cN_{R/E_2}[-1]) \\ & \cong
\RHom_{\Db(R)}((\cF_1\vert_{x_1}) \otimes \cO_R, (\cF_2\vert_{x_2}) \otimes \cN_{R/E_2}[-1])
\end{align*}
by adjunction for~$(\varepsilon_2^*,\varepsilon_{2*})$,
base change for~\eqref{eq:reex},
adjunction for~$(r_{2*},r_2^!)$,
the isomorphism~$\omega_{R/E_2} \cong \cN_{R/E_2}$,
and the fact that the compositions~$R \xrightarrow{\ r_k\ } E_k \xrightarrow{\ p_k\ } C_k$
coincide with~$R \xrightarrow\quad \Spec(\kk) \xrightarrow{\ x_k\ } C_k$.
On the other hand, $R \cong \P^1$ is a fiber of~$E_2 \to C_2$, hence~$\cN_{R/E_2} \cong \cO_R$, and~\eqref{eq:hom-phi12} follows. 

The isomorphisms~\eqref{eq:hom-pi-phi12} are proved in Lemma~\ref{lem:ac-blowup}.
\end{proof}

\begin{remark}
\label{rem:spherical-gluing}
Isomorphism~\eqref{eq:hom-phi12} shows that the subcategory generated by~$\Phi_1(\Db(C_1))$ and~$\Phi_2(\Db(C_2))$
in~$\Db(\hcXo)$ is equivalent to the gluing of~$\Db(C_1)$ and~$\Db(C_2)$
with the gluing object~$\cO_{(x_1,x_2)} \in \Db(C_1 \times C_2)$.
As we will see below, a mutation will transform it to the ideal point gluing.
\end{remark}

\begin{corollary}
We have
\begin{equation}
\label{eq:rhom-come1}
\RHom_{\Db(\hcXo)}(\cO_{\hcXo}(-E_1), \Phi_2(\cF_2)) \cong 
\RHom_{\Db(C_2)}(\cO_{C_2}(-x_2), \cF_2).
\end{equation} 
\end{corollary}

\begin{proof}
First of all, we observe that~$\RHom_{\Db(\hcXo)}(\cO_{\hcXo}(-E_1), \Phi_2(\cF_2)) \cong 
\RHom_{\Db(C_2)}(\Phi_{2}^*(\cO_{\hcXo}(-E_1)), \cF_2)$ by adjunction.
To compute~$\Phi_{2}^*(\cO_{\hcXo}(-E_1))$, we use the exact sequence~$0 \to \cO_{\hcXo}(-E_1) \to \cO_{\hcXo} \to \cO_{E_1} \to 0$,
and the isomorphisms
\begin{equation*}
\Phi_2^*(\cO_{\hcXo}) \cong \cO_{C_2}
\qquad\text{and}\qquad
\Phi_2^*(\cO_{E_1}) \cong \Phi_2^*(\Phi_1(\cO_{C_1})) \cong \cO_{x_2},
\end{equation*}
which follow from~\eqref{eq:hom-pi-phi12} and~\eqref{eq:hom-phi12}, respectively.
It follows that there is a distinguished triangle
\begin{equation*}
\Phi_2^*(\cO_{\hcXo}(-E_1)) \to \cO_{C_2} \to \cO_{x_2},
\end{equation*}
and it remains to check that the morphism~$\cO_{C_2} \to \cO_{x_2}$ is nonzero.

For this we consider the commutative diagram
\begin{equation*}
\xymatrix@C=3em{
\cO_{\hcXo}(-E_1) \ar[r] \ar[d] &
\cO_{\hcXo} \ar[r] \ar[d] &
\cO_{E_1} \ar[d]
\\
\Phi_2(\Phi_2^*(\cO_{\hcXo}(-E_1))) \ar[r] &
\cO_{E_2} \ar[r] &
\cO_R,
}
\end{equation*}
where the bottom line is obtained from the top by the functor~$\Phi_2 \circ \Phi_2^*$
(in particular, the map~$\cO_{E_2} \to \cO_R$ is obtained from~$\cO_{C_2} \to \cO_{x_2}$ by~$\Phi_2$)
and the vertical arrows are given by the unit of adjunction.
The two vertical arrows and the top horizontal arrow in the right square are obviously surjective,
hence so is the bottom arrow.
In particular, this arrow is nonzero, hence so is~$\cO_{C_2} \to \cO_{x_2}$, as required.

We conclude that~$\Phi_2^*(\cO_{\hcXo}(-E_1)) \cong \cO_{C_2}(-x_2)$,
and~\eqref{eq:rhom-come1} follows.
\end{proof}

We also make the following important observation about the spherical object~$\cO_L(-1)$.

\begin{lemma}
\label{lem:colm1}
There is a distinguished triangle
\begin{equation}
\label{eq:colm1}
\cO_L(-1) \to \Phi_1(\cO_{x_1}) \to \Phi_2(\cO_{x_2}),
\end{equation}
where the second arrow is the unique nonzero morphism.
\end{lemma}

\begin{proof}
By definition we have~$\Phi_1(\cO_{x_1}) \cong \cO_{R \cup L}$, $\Phi_2(\cO_{x_2}) \cong \cO_R$, 
hence the standard exact sequence
\begin{equation*}
0 \to \cO_L(-1) \to \cO_{R \cup L} \to \cO_R \to 0
\end{equation*}
gives the required distinguished triangle.
The uniqueness follows from~$\Hom(\cO_{R \cup L}, \cO_R) \cong \kk$.
\end{proof}

Now we modify~\eqref{eq:dbxp-1} by three mutations.
First, we mutate~$\cO_{\hcXo}$ one step to the right.
Using~\eqref{eq:hom-pi-phi12}, we see that~$\Phi_1^*(\cO_{\hcXo}) \cong \cO_{C_1}$, 
so that~$\Phi_1(\Phi_1^*(\cO_{\hcXo})) \cong \Phi_1(\cO_{C_1}) \cong \cO_{E_1}$,
and we conclude that the mutation functor~$\bR_{\Phi_1(\Db(C_1))}$ acts on~$\cO_{\hcXo}$ as follows:
\begin{equation*}
\bR_{\Phi_1(\Db(C_1))}(\cO_{\hcXo}) \cong
\Cone \Big(\cO_{\hcXo} \to \cO_{E_1} \Big)[-1] \cong
\cO_{\hcXo}(-E_1).
\end{equation*}
Thus, we obtain the following semiorthogonal decomposition
\begin{equation}
\label{eq:dbxp-2}
\Db(\hcXo) = \langle \pi^*(\cA_{\bcXo}), \Phi_1(\Db(C_1)), \cO_{\hcXo}(-E_1), \Phi_2(\Db(C_2)) \rangle.
\end{equation}

Next, we mutate the component~$\Phi_1(\Db(C_1))$ of~\eqref{eq:dbxp-2} one step to the right.
We obtain
\begin{equation}
\label{eq:dbxp-3}
\Db(\hcXo) = \langle \pi^*(\cA_{\bcXo}), \cO_{\hcXo}(-E_1), \Phi'_1(\Db(C_1)), \Phi_2(\Db(C_2)) \rangle,
\end{equation}
where the functor~$\Phi'_1 \colon \Db(C_1) \to \Db(\hcXo)$ is defined as the composition
\begin{equation*}
\Phi'_1 \coloneqq \bR_{\cO_{\hcXo}(-E_1)} \circ \Phi_1.
\end{equation*}

Finally, we mutate the component~$\pi^*(\cA_{\bcXo})$ of~\eqref{eq:dbxp-3} to the far right:
\begin{equation}
\label{eq:dbxp-4}
\Db(\hcXo) = 
\langle \cO_{\hcXo}(-E_1), \Phi'_1(\Db(C_1)), \Phi_2(\Db(C_2)), \pi^*(\cA_{\bcXo}) \otimes \cO_{\hcXo}(-K_{\hcXo}) \rangle.
\end{equation}

The next technical lemma will be used in the proof of Proposition~\ref{prop:gluing-geometric}.

\begin{lemma}
\label{lem:hom-phi1p-phi2}
For any object~$\cF_1 \in \Db(C_1)$ there is a distinguished triangle
\begin{equation}
\label{eq:phip-phi}
\Phi'_1(\cF_1) \to \Phi_1(\cF_1) \to \RHom_{\Db(C_1)}(\cF_1, \cO_{C_1}[-1])^\vee \otimes \cO_{\hcXo}(-E_1).
\end{equation}
Moreover, we have~$\Hom_{\Db(\hcXo)}(\Phi'_1(\cO_{C_1}), \Phi_2(\cO_{C_2}(-x_2))) = 0$.
\end{lemma}

\begin{proof}
Recall that~$\Phi'_1 = \bR_{\cO_{\hcXo}(-E_1)} \circ \Phi_1$ by definition
and that the right mutation functor~$\bR_{\cO_{\hcXo}(-E_1)}$ is defined by the distinguished triangle
\begin{equation*}
\bR_{\cO_{\hcXo}(-E_1)}(\cG) \to \cG \to \RHom(\cG, \cO_{\hcXo}(-E_1))^\vee \otimes \cO_{\hcXo}(-E_1).
\end{equation*}
Now we take~$\cG = \Phi_1(\cF_1)$.
Note that~$\RHom_{\Db(\hcXo)}(\Phi_1(\cF_1), \cO_{\hcXo}) = 0$ by~\eqref{eq:dbxp-1},
hence, applying the functor~$\RHom_{\Db(\hcXo)}(\Phi_1(\cF_1), -)$
to the exact sequence~$0 \to \cO_{\hcXo}(-E_1) \to \cO_{\hcXo} \to \cO_{E_1} \to 0$, we deduce that
\begin{equation*}
\RHom_{\Db(\hcXo)}(\Phi_1(\cF_1), \cO_{\hcXo}(-E_1)) \cong
\RHom_{\Db(\hcXo)}(\Phi_1(\cF_1), \Phi_1(\cO_{C_1})[-1]) \cong
\RHom_{\Db(C_1)}(\cF_1, \cO_{C_1}[-1]),
\end{equation*}
where we used an identification~$\cO_{E_1} \cong \Phi_1(\cO_{C_1})$ and full faithfulness of~$\Phi_1$.
This proves~\eqref{eq:phip-phi}.

Next, we note that~$\RHom_{\Db(C_1)}(\cO_{C_1}, \cO_{C_1}[-1]) \cong \kk[-1] \oplus \kk^{g_1}[-2]$;
therefore, for~$\cF_1 = \cO_{C_1}$ the triangle~\eqref{eq:phip-phi} takes the form
\begin{equation*}
\Phi'_1(\cO_{C_1}) \to \cO_{E_1} \to \cO_{\hcXo}(-E_1)[1] \oplus \cO_{\hcXo}(-E_1)^{\oplus g_1}[2].
\end{equation*}
Since the second arrow is the coevaluation morphism, its first component~$\cO_{E_1} \to \cO_{\hcXo}(-E_1)[1]$ 
corresponds to the unique non-trivial extension of~$\cO_{E_1}$ by~$\cO_{\hcXo}(-E_1)$, which is given by the exact sequence
\begin{equation*}
0 \to \cO_{\hcXo}(-E_1) \to \cO_{\hcXo} \to \cO_{E_1} \to 0.
\end{equation*}
This means that the above triangle can be rewritten as
\begin{equation*}
\Phi'_1(\cO_{C_1}) \to \cO_{\hcXo} \to \cO_{\hcXo}(-E_1)^{\oplus g_1}[2].
\end{equation*}
Since, on the other hand, $\Phi_2(\cO_{C_2}(-x_2)) \cong \cO_{E_2}(-R)$ 
and~$\Hom(\cO_{\hcXo}, \cO_{E_2}(-R)) = \rH^0(C_2, \cO_{C_2}(-x_2)) = 0$, 
the required equality follows.
\end{proof}

Now we check that the category~$\Db(C_1,C_2)$ (Definition~\ref{def:point-gluing}) 
is an admissible subcategory of~$\Db(\hcXo)$.

\begin{proposition}
\label{prop:gluing-geometric}
The subcategory generated by~$\Phi'_1(\Db(C_1))$ and~$\Phi_2(\Db(C_2))$ in~$\Db(\hcXo)$
is equivalent to the ideal point gluing~$\Db(C_1,C_2) = \Db(C_1,C_2; x_1,x_2)$ of~$\Db(C_1)$ and~$\Db(C_2)$. 

Under this equivalence the exotic exceptional object~$\rE \in \Db(C_1,C_2)$ defined in Theorem~\textup{\ref{thm:exotic-exceptional-curves}}
corresponds to the object~$\bT_{\cO_L(-1)}(\cO_{\hcXo}(-E_1))[1] \in \Db(\hcXo)$
obtained from~$\cO_{\hcXo}(-E_1)[1]$ by the spherical twist~$\bT_{\cO_L(-1)}$.
\end{proposition}

\begin{proof}
To prove the first part of the proposition, it is enough to relate the bimodule
\begin{equation*}
\rG(\cF_1,\cF_2) \coloneqq \RHom_{\Db(\hcXo)}(\Phi'_1(\cF_1), \Phi_2(\cF_2))
\end{equation*}
to the bimodule~\eqref{eq:ext-pg}.
Applying the functor~$\RHom_{\Db(\hcXo)}(-,\Phi_2(\cF_2))$ to~\eqref{eq:phip-phi} 
and using~\eqref{eq:hom-phi12} and~\eqref{eq:rhom-come1}, we obtain a distinguished triangle
\begin{multline*}
\RHom_{\Db(C_1)}(\cF_1, \cO_{C_1}[-1]) \otimes \RHom_{\Db(C_2)}(\cO_{C_2}(-x_2), \cF_2)
\\ \to
\RHom_{\Db(\kk)}(\cF_1\vert_{x_1}, \cF_2\vert_{x_2})[-1] \to
\RHom_{\Db(\hcXo)}(\Phi'_1(\cF_1), \Phi_2(\cF_2)).
\end{multline*}
Thus, $\RHom_{\Db(\hcXo)}(\Phi'_1(\cF_1), \Phi_2(\cF_2)) \cong \rH^\bullet(C_1 \times C_2, (\cF_1^\vee \boxtimes \cF_2) \otimes \urG)$,
where~$\urG$ fits into a distinguished triangle
\begin{equation}
\label{eq:rg-blowup}
\urG \to \cO_{C_1} \boxtimes \cO_{C_2}(x_2) \to \cO_{x_1} \boxtimes \cO_{x_2}.
\end{equation}
If the second arrow in~\eqref{eq:rg-blowup} is zero, 
then~$\cO_{C_1} \boxtimes \cO_{C_2}(x_2)$ is a direct summand of~$\urG$, hence
\begin{equation*}
\rH^0(C_1 \times C_2, \cO_{C_1} \boxtimes \cO_{C_2}(-x_2) \otimes \urG) \ne 0,
\end{equation*}
which means that~$\Hom_{\Db(\hcXo)}(\Phi'_1(\cO_{C_1}), \Phi_2(\cO_{C_2}(-x_2))) \ne 0$,
in contradiction to Lemma~\ref{lem:hom-phi1p-phi2}.
Thus, the second arrow in~\eqref{eq:rg-blowup} is nonzero, hence
\begin{equation*}
\urG \cong (\cO_{C_1} \boxtimes \cO_{C_2}(x_2)) \otimes \cI_{(x_1,x_2)}.
\end{equation*}
In particular, it is isomorphic to the gluing bimodule~$\cI_{(x_1,x_2)}$ of the ideal point gluing of~$\Db(C_1)$ and~$\Db(C_2)$
up to the autoequivalence of~$\Db(C_2)$ given by the twist with~$\cO_{C_2}(x_2)$,
hence the subcategory of~$\Db(\hcXo)$ generated by~$\Phi'_1(\Db(C_1))$ and~$\Phi_2(\Db(C_2))$
is equivalent to the ideal point gluing of~$\Db(C_1)$ and~$\Db(C_2)$, see Corollary~\ref{cor:gluing-auto}.

To prove the second part of the proposition, we apply the mutation functor~$\bR_{\cO_{\hcXo}(-E_1)}$ to triangle~\eqref{eq:colm1}.
Since~$\cO_{\hcXo}(-E_1)$ and~$\Phi_2(\Db(C_2))$ are semiorthogonal by~\eqref{eq:dbxp-2}, we obtain a distinguished triangle
\begin{equation}
\label{eq:re-new}
\bR_{\cO_{\hcXo}(-E_1)}(\cO_L(-1)) \to \Phi'_1(\cO_{x_1}) \to \Phi_2(\cO_{x_2}).
\end{equation}
On the other hand, it follows from~\eqref{eq:kl}, Serre duality on~$\hcXo$, and~\eqref{eq:ek-ell} that
\begin{multline*}
\RHom_{\Db(\hcXo)}(\cO_L(-1), \cO_{\hcXo}(-E_1)) \cong
\RHom_{\Db(\hcXo)}(\cO_{\hcXo}(-E_1), \cO_L(-1)[3])^\vee
\\\cong
\RHom_{\Db(L)}(\cO_L(1), \cO_L(-1)[3])^\vee \cong
\kk[-2].
\end{multline*}
Hence, the first term in~\eqref{eq:re-new} fits into a distinguished triangle
\begin{equation*}
\bR_{\cO_{\hcXo}(-E_1)}(\cO_L(-1)) \to \cO_L(-1) \to \cO_{\hcXo}(-E_1)[2],
\end{equation*}
which agrees with the triangle defining~$\bT_{\cO_L(-1)}(\cO_{\hcXo}(-E_1))$ up to shift (cf.~\eqref{eq:bt-ol}),
and we conclude that
\begin{equation}
\label{eq:bt-br}
\bT_{\cO_L(-1)}(\cO_{\hcXo}(-E_1)) \cong \bR_{\cO_{\hcXo}(-E_1)}(\cO_L(-1))[-1].
\end{equation}
In particular, since~$\cO_{\hcXo}(-E_1)$ is simple and~$\bT_{\cO_L(-1)}$ is an autoequivalence,
$\bR_{\cO_{\hcXo}(-E_1)}(\cO_L(-1))$ is simple, 
hence the second morphism in~\eqref{eq:re-new} is nonzero.
Finally, comparing~\eqref{eq:re-new} with the defining triangle~\eqref{eq:ee-triangle} of the exotic exceptional object~$\rE$
(note that~$\cO_{x_2} \otimes \cO_{C_2}(x_2) \cong \cO_{x_2}$, 
hence the additional autoequivalence of~$\Db(C_2)$ does not affect the exotic exceptional object),
we deduce the second part of the proposition.
\end{proof}

Combining the equivalence of Proposition~\ref{prop:gluing-geometric} with Definition~\ref{def:rpg} and~\eqref{eq:dbxp-4}, 
we obtain the following

\begin{corollary}
\label{cor:dbhcxo-ras}
There is a semiorthogonal decomposition
\begin{equation*}
\Db(\hcXo) = 
\langle \cO_{\hcXo}(-E_1), \bT_{\cO_L(-1)}(\cO_{\hcXo}(-E_1)), \RPG{C_1,C_2}, 
\pi^*(\cA_{\bcXo}) \otimes \cO_{\hcXo}(-K_{\hcXo}) \rangle,
\end{equation*}
where~$\RPG{C_1,C_2}$ is the reduced ideal point gluing of~$\Db(C_1)$ and~$\Db(C_2)$.
\end{corollary}

Now note that the semiorthogonal decomposition constructed in Corollary~\ref{cor:dbhcxo-ras}
fits into the setup of~\cite[Theorem~6.17]{KS24};
indeed, the first exceptional object~$\cO_{\hcXo}(-E_1)$ by~\eqref{eq:ek-ell} satisfies~\cite[(6.8)]{KS24},
and the second is obtained from it by the spherical twist with respect to~$\cO_L(-1)$.
Therefore, applying the theorem and using the terminology of~\cite{KS24}, we obtain the following

\begin{corollary}
\label{cor:dbcxo-ras}
There is a semiorthogonal decomposition
\begin{equation*}
\Db(\cXo) = 
\langle \varpi_*\cO_{\hcXo}(-E_1), \RPG{C_1,C_2}, 
\rho^*(\cA_{\bcXo}) \otimes \cO_{\cXo}(-K_{\cXo}) \rangle,
\end{equation*}
where the first component is a categorical ordinary double point
generated by the $\P^{\infty,2}$-object~$\varpi_*\cO_{\hcXo}(-E_1)$
providing a universal deformation absorption of singularities of~$\cXo$.
\end{corollary}

\subsection{The total space of the smoothing}

In this subsection we consider a smoothing of the 1-nodal curve~$C = C_1 \cup C_2$ 
and the induced smoothing of the 1-nodal threefold~$\cXo = \Bl_C(\bcXo)$
(constructed for an appropriate embedding~$C \hookrightarrow \bcXo$),
and extend the semiorthogonal decomposition of Corollary~\ref{cor:dbcxo-ras} 
to a linear semiorthogonal decomposition of a smoothing of~$\cXo$.

Recall that a smoothing of~$C$ over~$(B,o)$ is a flat projective morphism~$\cC \to B$
such that
\begin{itemize}
\item 
$\cC_o = C$,
\item 
$\cC_b$ is smooth for all~$b \ne o$, and
\item 
$\cC$ is a smooth surface.
\end{itemize}
To construct a threefold smoothing, we need the following simple observation.

\begin{lemma}
\label{lem:cc-bcx}
Assume the base field~$\kk$ is infinite.
For any smoothing~$\cC/B$ of a nodal curve~$C = \cC_o$ over a smooth pointed curve~$(B,o)$,
after passing to a Zariski neighborhood of~\mbox{$o \in B$},
there is a family~$\bcX/B$ of rational threefolds smooth and proper over~$B$
and a closed embedding~$\cC \hookrightarrow \bcX$ over~$B$.
\end{lemma}

\begin{proof}
We will construct an embedding into~$\bcX = \P^3 \times B$.

First, choosing a relatively very ample line bundle for the family of curves~$\cC/B$ and shrinking~$B$ if necessary,
we find an embedding~$\cC \hookrightarrow \P^N \times B$ for appropriate integer~$N \gg 0$.
Then we choose a sufficiently general linear subspace~$\P^{N-4} \subset \P^N$
so that it does not intersect the secant variety of the curve~$\cC_o \subset \P^N$,
nor the tangent spaces to~$\cC_o$ at the node.
Shrinking~$B$ further, we can assume that this subspace 
does not intersect the secant variety of each curve~$\cC_b$.
Then we consider the (constant) linear projection~$\P^N \times B \dashrightarrow \P^3 \times B$ out of~$\P^{N-4}$
and note that by construction the composition
\begin{equation*}
\cC \hookrightarrow \P^N \times B \dashrightarrow \P^3 \times B \eqqcolon \bcX
\end{equation*}
is a closed embedding.
\end{proof}

From now on we assume given a smoothing~$\cC/B$ of a 1-nodal curve~$C = C_1 \cup C_2$
and an embedding~\mbox{$\cC \hookrightarrow \bcX$} over~$B$
into a smooth and proper family of threefolds such that~$\cO_{\bcX_b}$ is exceptional for all~$b \in B$
as in Lemma~\ref{lem:cc-bcx},
and consider the flat proper family of threefolds
\begin{equation}
\label{eq:def-cx}
f \colon \cX \coloneqq \Bl_{\cC}(\bcX) \xrightarrow{\ \crho\ } \bcX \xrightarrow{\ \bar{f}\ } B,
\end{equation}
where~$\crho$ is the blowup morphism and~$f = \bar{f} \circ \crho$.
Applying Lemma~\ref{lem:x-mnf} to~$\bcXo \coloneqq \bcX_o$, we see that
\begin{itemize}
\item 
$\cX_o \cong \cXo = \Bl_{C}(\bcXo)$ is a maximally nonfactorial 1-nodal threefold,
\item 
$\cX_b = \Bl_{\cC_b}(\bcX_b)$ is a smooth threefold, and
\item 
$\cX$ is a smooth fourfold.
\end{itemize}
Thus, $\cX \to B$ is a smoothing of~$\cXo$.
We summarize the varieties and maps constructed so far on the following diagram
where the squares are Cartesian and the triangle is~\eqref{eq:diagram-x}:
\begin{equation}
\label{eq:cx-diagram}
\vcenter{
\xymatrix{
\hcXo \ar[r]^\varpi \ar[dr]_\pi &
\cXo \ar@{^{(}->}[r]^\io \ar[d]^{\rhoo} &
\cX \ar[d]^{\crho} \ar@/^3em/[dd]^f
\\
&
\bcXo \ar@{^{(}->}[r]^{\bar\io} \ar[d] &
\bcX \ar[d]^{\bar{f}}
\\
&
\{o\} \ar@{^{(}->}[r] &
B
}}.
\end{equation}
By the above assumption the structure sheaf~$\cO_{\bcX}$ is relatively exceptional, hence there is a $B$-linear semiorthogonal decomposition
\begin{equation}
\label{eq:dbbcx}
\Db(\bcX) = \langle \cA_{\bcX}, \bar{f}^*\Db(B) \rangle,
\qquad 
\text{where~$\cA_{\bcX} \coloneqq \Ker(\bar{f}_*)$}.
\end{equation}
In particular, $\cA_{\bcX} \subset \Db(\bcX)$ is a $B$-linear subcategory.
Now we can prove the main result of this section.

\begin{theorem}
\label{thm:main}
Let~$\cC/B$ be a smoothing of a $1$-nodal curve and let~$\cC \hookrightarrow \bcX$ be an embedding over~$B$
into a family of threefolds such that~$\cO_{\bcX_b}$ is exceptional for all~$b \in B$
and~$\bcX$ is smooth and proper over~$B$ as in Lemma~\textup{\ref{lem:cc-bcx}}.
Then there is a $B$-linear semiorthogonal decomposition
\begin{equation}
\label{eq:cd-cdp}
\Db(\cX) = \langle \io_*\varpi_*\cO_{\hcXo}(-E_1), \cD, \crho^*(\cA_{\bcX}) \otimes \cO_{\cX}(-K_{\cX/B}) \rangle,
\end{equation}
where~$\io_*\varpi_*\cO_{\hcXo}(-E_1)$ is an exceptional object 
and~$\cD$ is a triangulated category smooth and proper over~$B$ such that
\begin{equation*}
\cD_o \simeq \RPG{C_1,C_2}
\qquad\text{and}\qquad 
\cD_b \simeq \Db(\cO, \cC_b)
\quad\text{for~$b \ne o$}.
\end{equation*}
In other words, the central fiber of~$\cD/B$ is the reduced ideal point gluing of the curves~$C_1$ and~$C_2$, 
while all other fibers are the augmentations of the curves~$\cC_b$.
\end{theorem}

\begin{remark}
By Lemma~\ref{lem:cc-bcx}, Theorem~\ref{thm:main} applies to any smoothing of a 1-nodal curve if the field~$\kk$ is infinite.
We expect that Lemma~\ref{lem:cc-bcx}, and hence also Theorem~\ref{thm:main}, remain true also over finite fields.
\end{remark}

\begin{proof}
Recall from Corollary~\ref{cor:dbcxo-ras} that the sheaf~$\varpi_*\cO_{\hcXo}(-E_1)$ is a $\P^{\infty,2}$-object 
providing a universal deformation absorption of singularities of~$\cXo$.
Applying~\cite[Theorem~1.8]{KS24}, we conclude that its pushforward~$\io_*\varpi_*\cO_{\hcXo}(-E_1) \in \Db(\cX)$ is exceptional
and its orthogonal complement in~$\Db(\cX)$ is a smooth and proper $B$-linear category.

On the other hand, since~$\crho \colon \cX \to \bcX$ is the blowup of the smooth surface~$\cC$ in the smooth fourfold~$\bcX$,
the functor~$\crho^*$ is fully faithful, hence~$\crho^*(\cA_{\bcX}) \subset \Db(\cX)$ is an admissible $B$-linear subcategory.
So, to prove~\eqref{eq:cd-cdp} it is enough to show 
that~$\crho^*(\cA_{\bcX}) \otimes \cO_{\cX}(-K_{\cX/B})$ is semiorthogonal to~$\io_*\varpi_*\cO_{\hcXo}(-E_1)$;
indeed, then~$\cD$ can be defined as the intersection of the appropriate orthogonal complements
\begin{equation*}
\cD = {}^\perp\Big(\io_*\varpi_*\cO_{\hcXo}(-E_1)\Big) \cap \Big(\crho^*(\cA_{\bcX}) \otimes \cO_{\cX}(-K_{\cX/B})\Big)^\perp.
\end{equation*}
So, take any~$\cG \in \cA_{\bcX}$ and note that
\begin{equation*}
\RHom_{\Db(\cX)}(\crho^*(\cG) \otimes \cO_{\cX}(-K_{\cX/B}), \io_*\varpi_*\cO_{\hcXo}(-E_1)) \cong 
\RHom_{\Db(\hcXo)}(\varpi^*\io^*\crho^*(\cG) \otimes \cO_{\hcXo}(-K_{\hcXo}), \cO_{\hcXo}(-E_1)), 
\end{equation*}
where we used adjunction for~$(\varpi^*,\varpi_*)$ and~$(\io^*,\io_*)$ 
and an isomorphism~$\varpi^*\io^*\cO_{\cX}(-K_{\cX/B}) \cong \cO_{\hcXo}(-K_{\hcXo})$.
Furthermore, we have~$\varpi^*\io^*\crho^*(\cG) \cong \pi^*\bar\io^*(\cG)$ by commutativity of~\eqref{eq:cx-diagram}
and~$\bar\io^*(\cG) \in \cO_{\bcXo}^\perp = \cA_{\bcXo} \subset \Db(\bcXo)$ because~\eqref{eq:dbbcx} is $B$-linear.
Therefore, it follows from~\eqref{eq:dbxp-4} that the right side above vanishes,
hence the required semiorthogonality holds, 
and we obtain the semiorthogonal decomposition~\eqref{eq:cd-cdp}.

The category~$\cD$ defined by~\eqref{eq:cd-cdp} is $B$-linear (because the other components are),
and by~\cite[Theorem~1.5]{KS24} it is smooth and proper over~$B$, so it remains to identify its fibers~$\cD_o$ and~$\cD_b$.

To identify~$\cD_o$, we denote temporarily by~$\tcD$ the orthogonal complement to~$\io_*\varpi_*\cO_{\hcXo}(-E_1)$ in~\eqref{eq:cd-cdp},
so that~$\tcD = \langle \cD, \crho^*(\cA_{\bcX}) \otimes \cO_{\cX}(-K_{\cX/B}) \rangle$.
Then~\cite[Theorem~1.5]{KS24} shows that
\begin{equation*}
\tcD_o = {}^\perp\big(\varpi_*\cO_{\hcXo}(-E_1)\big) \subset \Db(\cXo).
\end{equation*}
Comparing this with Corollary~\ref{cor:dbcxo-ras}, we conclude that
\begin{equation*}
\tcD_o \simeq \langle \RPG{C_1,C_2}, \rho^*(\cA_{\bcXo}) \otimes \cO_{\cXo}(-K_{\cXo}) \rangle,
\end{equation*}
and it finally follows that~$\cD_o \simeq \RPG{C_1,C_2}$.

Similarly, for~$b \ne o$ after base change of~\eqref{eq:cd-cdp} along~$b \hookrightarrow B$, we obtain
\begin{equation*}
\Db(\cX_b) = \tcD_b = \langle \cD_b, \rho_{\cX_b}^*(\cA_{\bcX_b}) \otimes \cO_{\cX_b}(-K_{\cX_b}) \rangle,
\end{equation*}
where~$\cA_{\bcX_b} = \cO_{\bcX_b}^\perp \subset \Db(\bcX_b)$ and~$\rho_{\cX_b} \colon \cX_b \to \bcX_b$ is the restriction of~$\crho$.
Mutating the second component to the left, we obtain
\begin{equation*}
\Db(\cX_b) = \langle \rho_{\cX_b}^*(\cA_{\bcX_b}), \cD_b \rangle.
\end{equation*}
Comparing it with
\begin{equation*}
\Db(\cX_b) = 
\langle \rho_{\cX_b}^*(\Db(\bcX_b)), \Db(\cC_b) \rangle = 
\langle \rho_{\cX_b}^*(\cA_{\bcX_b}), \cO_{\cX_b}, \Db(\cC_b) \rangle, 
\end{equation*}
we conclude that~$\cD_b = \langle \cO_{\cX_b}, \Db(\cC_b) \rangle$, 
and by Lemma~\ref{lem:ac-blowup} we have~$\cD_b \simeq \Db(\cO,\cC_b)$,
as required.
\end{proof}

To conclude this section, we briefly discuss a relation
between the smooth and proper family of triangulated categories~$\cD$ constructed in Theorem~\ref{thm:main}
and the relative Jacobian for the family of curves~$\cC/B$.
Assume that~$\kk = \CC$ and (for simplicity) that the fibers of the smooth proper family of threefolds~$\bcX/B$
used in~\eqref{eq:def-cx} for the construction of~$\cX/B$ have trivial intermediate Jacobians
(e.g., $\bcX = \P^3 \times B$ as in Lemma~\ref{lem:cc-bcx}).
By~\cite[Conjecture~1.6]{KS25} under appropriate assumptions, a family of categories~$\cD/B$ smooth and proper over~$B$
should give rise to an abelian scheme~$\cJ/B$ 
fiberwise isomorphic to the family of intermediate Jacobians of the fibers of~$\cD/B$, 
i.e., so that~$\cJ_b \cong \Jac(\cD_b)$ for all~$b \in B$.
In our case,
\begin{equation*}
\Jac(\cD_o) \cong \Jac(C_1) \times \Jac(C_2)
\qquad\text{and}\qquad
\Jac(\cD_b) \cong \Jac(\cC_b)
\qquad\text{for~$b \ne o$}
\end{equation*}
by Theorem~\ref{thm:main} and Propositions~\ref{prop:invariants-rpg}\ref{it:jac-rpg}, \ref{prop:invariants-ac}\ref{it:jac-ac},
and such an abelian scheme can be constructed ad hoc as~$\cJ \coloneqq \Pic^0(\cC/B)$.
Thus, the family of categories~$\cD/B$ can be considered as a categorification
of this family~$\cJ/B$ of principally polarized abelian varieties.

\appendix

\section{BN-modifications of genus~$4$ curves and $1$-nodal cubic threefolds}
\label{sec:cubic-3}

In this section we work over an algebraically closed field~$\kk$.
If~$Y \subset \P^4$ is a smooth cubic threefold, there is a semiorthogonal decomposition
\begin{equation}
\label{eq:dby}
\Db(Y) = \langle \cB_Y, \cO_Y(-1), \cO_Y \rangle.
\end{equation} 
The category~$\cB_Y$ defined by this decomposition has many interesting properties;
for instance, it is a fractional Calabi--Yau category with~$\bS_{\cB_Y}^3 \cong [5]$
(see~\cite[Corollary~4.4]{K04} or~\cite[Corollary~4.1]{K19}).
If~$Y$ is a $1$-nodal cubic threefold, the semiorthogonal decomposition~\eqref{eq:dby} is still defined.
The goal of this section is to relate the corresponding category~$\cB_Y$ to a BN-modification of a curve, 
see Definition~\ref{def:bn-modification}.

So, let~$Y$ be a $1$-nodal cubic threefold and let~$P \in Y$ be the node.
Consider the blowup
\begin{equation*}
X \coloneqq \Bl_P(Y) \xrightarrow{\ \pi\ } Y.
\end{equation*}
Then~$X$ is a resolution of singularities of~$Y$ and its exceptional divisor is a smooth quadric surface~$Q \subset X$.
However, this resolution is not categorically minimal:
by~\cite[Theorem~4.4 and Proposition~4.7]{K08b}
the sheaves~$(\cO_Q,\cO_Q(1,0))$ form an exceptional pair in~$\Db(X)$
and their orthogonal complement provides a strongly crepant categorical resolution for~$Y$.
Combining this observation with~\eqref{eq:dby}, we obtain a semiorthogonal decomposition
\begin{equation}
\label{eq:sod-final}
\Db(X) = \langle \tcB_Y, \cO_X(-H), \cO_X, \cO_Q, \cO_Q(1,0) \rangle,
\end{equation} 
where~$H$ is the pullback of the hyperplane class of~$Y$ 
and the category~$\tcB_Y$ provides a strongly crepant categorical resolution for~$\cB_Y$.
Recall Definition~\ref{def:bn-modification} and Lemma~\ref{lem:ac-4}.

\begin{proposition}
\label{prop:tbcy}
Let~$Y$ be a $1$-nodal cubic threefold.
The category~$\tcB_Y$ defined by~\eqref{eq:sod-final}
is equivalent to a BN-modification of a curve~$C$ of genus~$4$ with respect to a trigonal line bundle.
\end{proposition}

\begin{proof}
The linear projection~$\rho \colon \Bl_P(Y) \to \P^3$ out of~$P$ induces an isomorphism
\begin{equation*}
\Bl_P(Y) \cong \Bl_C(\P^3),
\end{equation*}
where~$C \subset \P^3$ is a smooth complete intersection of the smooth quadric~$\rho(Q)$ and a cubic in~$\P^3$,
i.e., a canonically embedded curve of genus~$\g(C) = 4$.
We consider the diagram of blowups
\begin{equation*}
\xymatrix@C=3em{
& 
Q \ar@{^{(}->}[r] \ar[dl] &
\Bl_P(Y) \ar[dl]^{\pi} &
X \ar@{=}[l]  &
\Bl_C(\P^3) \ar@{=}[l] \ar[dr]_{\rho} &
E \ar@{_{(}->}[l]_-i \ar[dr]^p
\\
P \ar@{^{(}->}[r] &
Y &&&& 
\P^3 &
C \ar@{_{(}->}[l]
},
\end{equation*}
where~$E$ is the exceptional divisor of~$\rho$.
We denote by~$h$ the pullback to~$X$ of the hyperplane class of~$\P^3$.
Then we have the following equalities in the Picard group~$\Pic(X)$:
\begin{equation}
\label{eq:HQ-hE}
H = 3h - E,
\qquad 
Q = 2h - E,
\end{equation}
and the canonical class formula for the blowups~$\pi$ and~$\rho$ gives the equalities
\begin{equation}
\label{eq:kx}
K_X = Q - 2H = E - 4h.
\end{equation}
Furthermore, the blowup formula for~$\rho \colon X \to \P^3$ gives the following semiorthogonal decomposition
\begin{equation}
\label{eq:sod-0}
\Db(X) = \langle \Phi(\Db(C)), \cO_X, \cO_X(h), \cO_X(2h), \cO_X(3h) \rangle,
\qquad\text{where}\quad
\Phi(\cF) \coloneqq i_*(p^*(\cF)) \otimes \cO_X(E).
\end{equation}
Now we modify~\eqref{eq:sod-0} by a couple of mutations.

First, we mutate~$\cO_X(h)$, $\cO_X(2h)$, and~$\cO_X(3h)$ to the far left.
By~\eqref{eq:kx} we obtain
\begin{equation*}
\Db(X) = \langle \cO_X(E-3h), \cO_X(E-2h), \cO_X(E-h), \Phi(\Db(C)), \cO_X \rangle.
\end{equation*}

Next, we mutate~$\cO_X(E - 2h)$, $\cO_X(E-h)$, and~$\Phi(\Db(C))$ to the right of~$\cO_X$.
Since~$2h - E = Q$ by~\eqref{eq:HQ-hE} and~$Q$ is the exceptional divisor of~$\pi$, we have~$\rH^\bullet(X,\cO_X(2h - E)) = \kk$,
hence the mutation of~$\cO_X(E - 2h)$ is realized by the exact sequence
\begin{equation*}
0 \to \cO_X(E - 2h) \to \cO_X \to \cO_Q \to 0.
\end{equation*}
Taking also into account the equality~$E - 3h = -H$ (see~\eqref{eq:HQ-hE}), we obtain
\begin{equation}
\label{eq:sod-p3-2}
\Db(X) = \langle \cO_X(-H), \cO_X, \cO_Q, \bR_{\cO_X}(\cO_X(E-h)), \bR_{\cO_X}(\Phi(\Db(C))) \rangle.
\end{equation}

On the other hand, if we mutate~$\tcB_Y$ in~\eqref{eq:sod-final} to the far right, by~\eqref{eq:kx} we obtain
\begin{equation}
\label{eq:sod-y-2}
\Db(X) = \langle \cO_X(-H), \cO_X, \cO_Q, \cO_Q(1,0), \tcB_Y(2H - Q) \rangle.
\end{equation} 
Now we observe that the first three components in decompositions~\eqref{eq:sod-p3-2} and~\eqref{eq:sod-y-2} agree.
Therefore, their orthogonal complements coincide, i.e., we have
\begin{equation}
\label{eq:ac-x}
\langle \bR_{\cO_X}(\cO_X(E-h)), \bR_{\cO_X}(\Phi(\Db(C))) \rangle =
\langle \cO_Q(1,0), \tcB_Y(2H - Q) \rangle. 
\end{equation}
We will show that this category is equivalent to the augmentation~$\Db(\cO, C)$ in such a way
that the exceptional objects~$\bR_{\cO_X}(\cO_X(E-h))$ and~$\cO_Q(1,0)$ in the left and right sides
correspond to the canonical exceptional object~$\cE \in \Db(\cO,C)$
and a BN-exceptional object~$\cE_\cL \in \Db(\cO,C)$, respectively.

To start with, we compute~$\Ext^\bullet(\bR_{\cO_X}(\cO_X(E-h)), \bR_{\cO_X}(\Phi(\cF)))$.
Since the mutation functor~$\bR_{\cO_X}$ defines an equivalence~$\cO_X^\perp \to {}^\perp\cO_X$, we have an isomorphism
\begin{equation*}
\Ext^\bullet(\bR_{\cO_X}(\cO_X(E-h)), \bR_{\cO_X}(\Phi(\cF))) \cong \Ext^\bullet(\cO_X(E-h), \Phi(\cF)).
\end{equation*}
Now, recalling from~\eqref{eq:sod-0} the definition of~$\Phi$, we find
\begin{multline*}
\Ext^\bullet(\cO_X(E-h), \Phi(\cF)) =
\Ext^\bullet(\cO_X(E-h), i_*p^*(\cF) \otimes \cO_X(E)) \cong
\Ext^\bullet(\cO_X(-h), i_*p^*(\cF)) \\ \cong
\Ext^\bullet(i^*\cO_X(-h), p^*(\cF)) \cong
\Ext^\bullet(p^*\omega_C^{-1}, p^*(\cF)) \cong
\Ext^\bullet(\omega_C^{-1}, \cF) {} \cong
\rH^\bullet(C, \cF \otimes \omega_C).
\end{multline*}
This proves that the category~\eqref{eq:ac-x} is equivalent to the gluing of~$\Db(\kk)$ 
and~$\Db(C)$ with the gluing object~$\omega_C$, hence to the augmented curve~$C$, see Corollary~\ref{cor:gluing-auto}.
Moreover, under this equivalence the object~$\bR_{\cO_X}(\cO_X(E-h))$
corresponds to the canonical exceptional object~$\cE$ of the augmentation.

To identify the object~$\cO_Q(1,0)$, we decompose it with respect to the first semiorthogonal decomposition of~\eqref{eq:ac-x}.
Since~$\bR_{\cO_X}$ is an equivalence~$\cO_X^\perp \to {}^\perp\cO_X$ with inverse~$\bL_{\cO_X}$,
this is equivalent to the decomposition of
\begin{equation*}
\bL_{\cO_X}(\cO_Q(1,0)) \cong
\Cone \Big(\cO_X^{\oplus 2} \to \cO_Q(1,0) \Big)
\end{equation*}
with respect to the semiorthogonal decomposition~$\langle \cO_X(E-h), \Phi(\Db(C)) \rangle$.

First, we compute the projection of~$\bL_{\cO_X}(\cO_Q(1,0))$ to the component~$\Db(C)$; for this we apply the right adjoint~$\Phi^!$ of~$\Phi$.
It follows from~\eqref{eq:sod-0} that~$\Phi^!(\cF) \cong p_*(i^!(\cF \otimes \cO_X(-E))) \cong p_*i^*(\cF)[-1]$.
Therefore,
\begin{equation*}
\Phi^!(\cO_X) \cong p_*i^*\cO_X[-1] \cong p_*\cO_E[-1] \cong \cO_C[-1].
\end{equation*}
Similarly, since~$Q \cap E = C$ and the restrictions of~$\cO_Q(1,0)$ and~$\cO_Q(0,1)$ to~$C$ 
are the trigonal line bundles on~$C$, which we denote by~$\cL_1$ and~$\cL_2$, respectively,
we have
\begin{equation*}
\Phi^!(\cO_Q(1,0)) \cong p_*i^*\cO_Q(1,0)[-1] \cong p_*\cL_1[-1] \cong \cL_1[-1].
\end{equation*}
We conclude that
\begin{equation*}
\Phi^!(\bL_{\cO_X}(\cO_Q(1,0))) \cong
\Cone \Big(\Phi^!(\cO_X^{\oplus 2}) \to \Phi^!(\cO_Q(1,0)) \Big) \cong
\Cone \Big(\cO_C^{\oplus 2} \to \cL_1 \Big)[-1] \cong \cL_1^{-1}.
\end{equation*}
Taking into account that the equivalence of the category~\eqref{eq:ac-x} with~$\Db(\cO,C)$ includes a twist by~$\omega_C$,
we finally find that the component of~$\cO_Q(1,0)$ in~$\Db(C)$ is~$\cL_1^{-1} \otimes \omega_C \cong \cL_2$, 
the second trigonal bundle.

Next, we compute the projection of~$\bL_{\cO_X}(\cO_Q(1,0))$ to the first component of~$\langle \cO_X(E-h), \Phi(\Db(C)) \rangle$,
i.e., we compute the space~$\Ext^\bullet(\bL_{\cO_X}(\cO_Q(1,0)), \cO_X(E-h))$.
Using Serre duality and~\eqref{eq:kx}, we compute
\begin{equation*}
\Ext^\bullet(\cO_X, \cO_X(E-h)) \cong
\rH^{3-\bullet}(X, \cO_X(-3h))^\vee \cong
\rH^{3-\bullet}(\P^3, \cO_{\P^3}(-3))^\vee = 0.
\end{equation*}
Similarly,
\begin{equation*}
\Ext^\bullet(\cO_Q(1,0), \cO_X(E-h)) \cong
\rH^{3-\bullet}(X, \cO_Q(1,0) \otimes \cO_X(-3h))^\vee \cong
\rH^{3-\bullet}(Q, \cO_{Q}(-2,-3))^\vee = \kk^2[-1].
\end{equation*}
We conclude that~$\Ext^\bullet(\bL_{\cO_X}(\cO_Q(1,0)), \cO_X(E-h)) \cong \kk^{\oplus 2}[-1]$.

Altogether, this means that~$\cO_Q(1,0)$ corresponds to an object of the form~$(\kk^2, \cL_2, \phi)[1]$ in~$\Db(\cO,C)$,
for some morphism~$\phi \colon \kk^2 \to \rH^0(C, \cL_2)$.
But since~$\cO_Q(1,0)$ is exceptional, the argument of Proposition~\ref{prop:bnp-exceptional}
shows that~$\rH^0(\phi)$ is an isomorphism, hence~$(\kk^2, \cL_2, \phi)$ is the BN-exceptional object~$\cE_{\cL_2}$.

Finally, since the category~$\tcB_Y$ is equivalent to the orthogonal complement of~$\cO_Q(1,0)$ in~\eqref{eq:ac-x},
i.e., to the orthogonal complement of the BN-exceptional object~$\cE_{\cL_2}$ in~$\Db(\cO,C)$,
it is equivalent to the BN-modification of~$\Db(C)$.
\end{proof}

Proposition~\ref{prop:tbcy} gives a description of the categorical resolution~$\tcB_Y$ of~$\cB_Y$.
Now, we deduce a description of~$\cB_Y$.
Recall from Lemma~\ref{lem:ac-4} that the exceptional objects~$\cE_{\cL_1}$ and~$\cE_{\cL_2}$ form a spherical pair.

\begin{corollary}
Let~$Y$ be a $1$-nodal cubic threefold.
The category~$\cB_Y$ defined by~\eqref{eq:dby}
is equivalent to the Verdier quotient of the BN-modification~${}^\perp\cE_{\cL_2} \subset \Db(\cO, C)$
by the $3$-spherical object~\mbox{$\Cone \big(\cE_{\cL_1}[-2] \to \cE_{\cL_2}\big)$}.
\end{corollary}

\begin{proof}
Recall from~\cite[Theorem~5.8, Lemma~5.10, and~(5.13)]{KS24} that the category~$\Db(Y)$ 
is equivalent to the Verdier quotient of~$\langle \cO_Q, \cO_Q(1,0) \rangle^\perp \subset \Db(X)$ by the 3-spherical object
\begin{equation*}
\rK \coloneqq \Cone \Big(\cO_Q(-1,0) \to \cO_Q(0,-1)[2] \Big)
\end{equation*}
contained in~$\tcB_Y$.
Using~\eqref{eq:dby} and~\eqref{eq:sod-final}, we conclude that~$\cB_Y$ is the Verdier quotient of~$\tcB_Y$ by~$\rK$.

On the other hand, Proposition~\ref{prop:tbcy}
gives an equivalence
\begin{equation*}
\langle \cO_Q(1,0), \tcB_Y(2H - Q) \rangle =
{}^\perp\big\langle \cO_X(-H), \cO_X, \cO_Q \big\rangle \xrightarrow{\quad\simeq\quad} \Db(\cO, C)
\end{equation*}
such that the sheaves~$\cO_Q(0,1)$ and~$\cO_Q(1,0)$
correspond to the BN-exceptional objects~$\cE_{\cL_1}[1]$ and~$\cE_{\cL_2}[1]$ in~$\Db(\cO,C)$
and the subcategory~$\tcB_Y(2H - Q)$ corresponds to the BN-modification~${}^\perp\cE_{\cL_2}$.
Using an isomorphism~$\cO_Q(0,-1) \otimes \cO_X(2H - Q) \cong \cO_Q(1,0)$
and composing the above equivalence with the equivalence
\begin{equation*}
\langle \cO_Q(0,-1), \tcB_Y \rangle \xrightarrow{\quad\simeq\quad}
\langle \cO_Q(1,0), \tcB_Y(2H - Q) \rangle
\end{equation*}
induced by the~$\cO_X(2H - Q)$-twist of~$\Db(X)$,
we obtain an equivalence~\mbox{$\langle \cO_Q(0,-1), \tcB_Y \rangle \simeq \Db(\cO,C)$}
such that the sheaves~$\cO_Q(-1,0)$ and~$\cO_Q(0,-1)$
correspond to~$\cE_{\cL_1}[1]$ and~$\cE_{\cL_2}[1]$,
hence the object~$\rK$ corresponds (up to shift) to~$\Cone \big(\cE_{\cL_1}[-2] \to \cE_{\cL_2} \big)$.
\end{proof}

\section{The gluing functor for intersections of quadrics of general type}
\label{sec:quadrics}

In this section we work over a field~$\kk$ of characteristic not equal to~2.

Let~$W$ and~$V$ be vector spaces with~$\dim(W) = n$ and~$\dim(V) = 2n - 1$
and let~$\bq \colon W \hookrightarrow \Sym^2V^\vee$ be a linear embedding such that
\begin{equation*}
X \coloneqq \bigcap_{w \in W} Q_w \subset \P(V)
\end{equation*}
(where~$Q_w \subset \P(V)$ is the quadric with equation~$\bq(w)$)
is a smooth subvariety of dimension~$n - 2$.
Thus, $X$ is a smooth complete intersection of~$n$ quadrics in~$\P^{2n-2}$,
hence it is a variety of general type with~$\omega_X \cong \cO_{\P(V)}(1)\vert_X$.
Following~\cite[Section~5]{K08} (see also~\cite[(12)]{K08}), we consider
\begin{equation*}
\Cliff_0 = 
\cO_{\P(W)} \oplus 
\big(\wedge^2V \otimes \cO_{\P(W)}(-1)\big) \oplus 
\dots \oplus 
\big(\wedge^{2n-2}V \otimes \cO_{\P(W)}(1-n)\big),
\end{equation*}
the sheaf of even Clifford algebras on~$\P(W)$ corresponding to~$X$. 
Here is the main result of this section.

\begin{proposition}
\label{prop:gluing-quadrics}
For any~$n \ge 3$ the derived category~$\Db(\P(W), \Cliff_0)$ of sheaves of coherent~$\Cliff_0$-modules on~$\P(W)$
is equivalent to the gluing of~$\Db(X)$ and~$\Db(\kk)$
with the gluing object~$\urG \cong \cO_X$.

In particular, if~$n = 3$, so that~$X \subset \P^4$ is a smooth canonical curve of genus~$g = 5$,
the category~$\Db(\P^2, \Cliff_0)$ is equivalent to the augmentation~$\Db(\cO,X)$ of~$X$.
\end{proposition}

\begin{proof}
By~\cite[Theorem~5.5]{K08} the object~$\Cliff_0 \in \Db(\P(W), \Cliff_0)$ is exceptional 
and there is a semiorthogonal decomposition
\begin{equation}
\label{eq:sod-pw-cl}
\Db(\P(W), \Cliff_0) = 
\langle \Phi(\Db(X)), \Cliff_0 \rangle,
\end{equation}
where~$\Phi \colon \Db(X) \hookrightarrow \Db(\P(W), \Cliff_0)$ is a fully faithful embedding.
To prove the first part of the proposition, we need to compute the gluing object~$\Phi^!(\Cliff_0) \in \Db(X)$;
to achieve this, we use a description of the functor~$\Phi^!$ from~\cite{K08}, which we recall below.

Let~$i \colon Q \hookrightarrow \P(W) \times \P(V)$ be the embedding of the divisor of bidegree~$(1,2)$ with the equation~$\bq$
(this is the universal quadric through~$X$),
and let~$f \colon Q \to \P(W)$ be the natural projection.
By~\cite[Section~4]{K08} there is a sheaf of~$f^*\Cliff_0$-modules~$\cS$ on~$Q$ 
which fits into an exact sequence of~$\Cliff_0 \boxtimes \cO_{\P(V)}$-modules
\begin{equation}
\label{eq:spinor}
0 \to \Cliff_{-1} \boxtimes \cO_{\P(V)}(-1) \xrightarrow{\ \updelta\ } \Cliff_0 \boxtimes \cO_{\P(V)} \to i_*\cS \to 0,
\end{equation}
where~$\Cliff_{-1} = 
\big(V \otimes \cO_{\P(W)}(-1)\big) \oplus 
\big(\wedge^3V \otimes \cO_{\P(W)}(-2)\big) \oplus 
\dots \oplus 
\big(\wedge^{2n-1}V \otimes \cO_{\P(W)}(-n)\big)$
is the sheaf of odd parts of the Clifford algebra (see~\cite[(14) and~(15)]{K08})
and~$\updelta$ is induced by Clifford multiplication, see~\cite[(23)]{K08}.
Finally, there is an embedding~$\P(W) \times X \subset Q$ and an isomorphism
\begin{equation*}
\Phi^!(\cF) \cong q_*(p^*\cF \otimes_{p^*\Cliff_0} \cS\vert_{\P(W) \times X}),
\end{equation*}
where~$p \colon \P(W) \times X \to \P(W)$ and~$q \colon \P(W) \times X \to X$ are the projections.
Therefore, 
\begin{equation}
\label{eq:alpha-shriek}
\Phi^!(\Cliff_0) \cong q_*(p^*\Cliff_0 \otimes_{p^*\Cliff_0} \cS\vert_{\P(W) \times X}) \cong q_*(\cS\vert_{\P(W) \times X}).
\end{equation}
So, we only need to check that~$q_*(\cS\vert_{\P(W) \times X}) \cong \cO_X$.

Restricting~\eqref{eq:spinor} to~$\P(W) \times X$, we obtain an exact sequence
\begin{equation}
\label{eq:spinor-restricted}
0 \to 
\cS\vert_{\P(W) \times X}(-1,-2) \to 
\Cliff_{-1} \boxtimes \cO_X(-1) \xrightarrow{\ \updelta\ } 
\Cliff_0 \boxtimes \cO_X \to 
\cS\vert_{\P(W) \times X} \to 
0.
\end{equation}
By definition of the Clifford multiplication, the restriction~$\updelta_v$ of~$\updelta$ 
to~$\P(W) \times [v] \subset \P(W) \times X$ for each point~$[v] \in X \subset \P(V)$
is the sum for~$k = 1, \dots, n$ of the maps
\begin{equation*}
\wedge^{2k-1}V \otimes \cO_{\P(W)}(-k) \xrightarrow{\ \bq(v)\ } \wedge^{2k-2}V \otimes \cO_{\P(W)}(1-k)
\end{equation*}
given by~$\bq(v) \in W^\vee \otimes V^\vee$ and the wedge product maps
\begin{equation}
\label{eq:wedge}
\wedge^{2k-1}V \otimes \cO_{\P(W)}(-k) \xrightarrow{\ - \wedge v\ } \wedge^{2k}V \otimes \cO_{\P(W)}(-k).
\end{equation}
In particular, we see that~$\updelta$ is compatible with the filtrations defined by
\begin{align*}
\bF_k(\Cliff_{-1} \boxtimes \cO_X(-1)) &\coloneqq \bigoplus_{i \le k} \wedge^{2i-1}V \otimes \cO_{\P(W)}(-i) \boxtimes \cO_X(-1),
\\
\bF_k(\Cliff_0 \boxtimes \cO_X) &\coloneqq \bigoplus_{i \le k} \wedge^{2i}V \otimes \cO_{\P(W)}(-i) \boxtimes \cO_X
\end{align*}
(this compatibility also follows from the obvious equality~$\Hom(\cO_{\P(W)}(-i), \cO_{\P(W)}(-j)) = 0$ for~$i < j$).

Now we consider the spectral sequence of the filtered complex~$\Cliff_{-1} \boxtimes \cO_X(-1) \xrightarrow{\ \updelta\ } \Cliff_0 \boxtimes \cO_X$.
Its zeroth differential~$\bd_{k,-1}^0 \colon \bE_{k,-1}^0 \to \bE_{k,0}^0$ 
is a relative version of the map~\eqref{eq:wedge};
more precisely, it is the map 
\begin{multline}
\label{eq:wedge-rel}
\wedge^{2k-1}V \otimes \cO_{\P(W)}(-k) \boxtimes \cO_X(-1) = 
\gr_k^\bF(\Cliff_{-1} \boxtimes \cO_{X}(-1))
\\ \xrightarrow{\quad}
\gr_k^\bF(\Cliff_0 \boxtimes \cO_{X}) = 
\wedge^{2k}V \otimes \cO_{\P(W)}(-k) \boxtimes \cO_X
\end{multline}
induced by the tautological embedding~$\cO_X(-1) \hookrightarrow V \otimes \cO_X$ and wedge product in~$\wedge^\bullet V$.
In other words, the differentials~$\bd_{k,-1}^0$ are obtained from the maps in the Koszul complex of~$\P(V)$
\begin{equation*}
0 \to 
\cO_{\P(V)} \to 
V \otimes \cO_{\P(V)}(1) \to 
\dots \to 
\wedge^{2n-2}V \otimes \cO_{\P(V)}(2n-2) \to 
\wedge^{2n-1}V \otimes \cO_{\P(V)}(2n-1) \to 0
\end{equation*}
by a twist, restriction to~$X$, and box tensor product with~$\cO_{\P(W)}(-k)$.
It follows that the kernel and cokernel of~\eqref{eq:wedge-rel}
are isomorphic to
\begin{align*}
\bE_{k,-1}^1 &= 
\cO_{\P(W)}(-k) \boxtimes \wedge^{2k-2}(V \otimes \cO_X / \cO_X(-1)) \otimes \cO_X(-2)
\qquad\text{and}\\
\bE_{k,0}^1 &= 
\cO_{\P(W)}(-k) \boxtimes \wedge^{2k}(V \otimes \cO_X / \cO_X(-1)),
\end{align*}
respectively.
In particular, we see that the sheaves~$\bigoplus_k \bE_{k,-1}^1$ and~$\bigoplus_k \bE_{k,0}^1$ 
are locally free sheaves of the same rank~$2^{2n-3}$.
On the other hand, by~\eqref{eq:spinor-restricted} the spectral sequence 
converges to the associated graded sheaves~$\bigoplus_k \bE_{k,-1}^\infty$ and~$\bigoplus_k \bE_{k,0}^\infty$
of appropriate filtrations on the sheaves~$\cS\vert_{\P(W) \times X}(-1,-2)$ and~$\cS\vert_{\P(W) \times X}$, respectively, 
which are also locally free of the same rank~$2^{2n-3}$ by~\cite[Lemma~4.7]{K08}.
Therefore, the spectral sequence degenerates at the first page, 
and we conclude that~$\cS\vert_{\P(W) \times X}$ has a filtration with associated graded sheaves
\begin{equation*}
\gr_k^\bF(\cS\vert_{\P(W) \times X}) \cong
\bE_{k,0}^\infty \cong 
\bE_{k,0}^1 \cong 
\cO_{\P(W)}(-k) \boxtimes \wedge^{2k}(V \otimes \cO_X / \cO_X(-1))
\qquad\text{for~$0 \le k \le n - 1$}.
\end{equation*}
Since~$\rH^\bullet(\P(W), \cO_{\P(W)}(-k)) = 0 $ for~$1 \le k \le n-1$,
we have~$q_*(\gr_k^\bF(\cS\vert_{\P(W) \times X})) = 0$ and therefore
\begin{equation*}
q_*(\cS\vert_{\P(W) \times X}) \cong
q_*(\gr_0^\bF(\cS\vert_{\P(W) \times X})) \cong
q_*(\cO_{\P(W)} \boxtimes \cO_X) \cong 
\cO_X.
\end{equation*}
Combining this with~\eqref{eq:alpha-shriek}, we conclude that 
the gluing object of~\eqref{eq:sod-pw-cl} is isomorphic to~$\cO_X$.

In the case where~$n = 3$, we proved that~$\Db(\P^2, \Cliff_0)$ is equivalent to the gluing of~$\Db(X)$ and~$\Db(\kk)$,
where~$X$ is a curve of genus~$5$ and the gluing object is isomorphic to the structure sheaf~$\cO_X$ of~$X$.
Applying Remark~\ref{rem:ac-lb}, we finally conclude that this category is equivalent to the augmentation of~$X$.
\end{proof}

\end{document}